\documentclass[preprint,12pt]{elsarticle}
\usepackage{hyperref}

\usepackage{url}            
\usepackage{booktabs}       
\usepackage{amsfonts}       
\usepackage{nicefrac}       
\usepackage{microtype}      
\usepackage{lipsum}
\usepackage{array,multirow}
\usepackage{blindtext}
\usepackage{graphicx}
\usepackage{xcolor}
\usepackage{amsmath}
\usepackage{mathtools}
\usepackage{comment}
\usepackage{amsfonts}
\usepackage{blkarray}
\usepackage{bigstrut}
\usepackage{blindtext}
\usepackage[linesnumbered,boxed]{algorithm2e}
\usepackage{amsthm}
\SetKwInput{KwInput}{Input}
\SetKwInput{KwOutput}{Output}
\usepackage[tight,footnotesize]{subfigure}
\usepackage{caption}
\usepackage{booktabs}
\usepackage{comment}
\usepackage{cleveref}

\theoremstyle{definition}
\newtheorem{thm}{Theorem}[section]
\newtheorem{assumption}{Assumption}[section]

\newtheorem{defn}[thm]{Definition}
\newtheorem{prop}[thm]{Proposition}
\newtheorem{cor}[thm]{Corollary}
\newtheorem{example}[thm]{Example}

\newtheorem{proposition}[thm]{Proposition}

\usepackage{tikz}
\usepackage{tikz-cd}
\usetikzlibrary{shapes,backgrounds}
\usepackage[utf8]{inputenc}
\usepackage{amsmath,xcolor,graphicx,array,multicol,manfnt}
\usepackage{tkz-euclide}
\usepackage{float}

\usepackage{listofitems}

\usepackage{amssymb}
\usepackage{amsthm}

\newcommand{\VV}{\mathbf{V}}

\newcommand{\xx}{\mathbf{x}}
\newcommand{\yy}{\mathbf{y}}

\newcommand{\pp}{\mathbf{p}}
\renewcommand{\aa}{\mathbf{a}}
\newcommand{\bb}{\mathbf{b}}

\newcommand{\XX}{\mathbf{X}}
\newcommand{\R}{\mathbf{R}}

\newcommand{\I}{\mathbf{I}}
\newcommand{\tran}{\mathbf{t}}
\newcommand{\al}{\boldsymbol{\alpha}}
\newcommand{\be}{\boldsymbol{\beta}}

\newcommand{\rr}{\mathbf{r}}
\renewcommand{\t}{\mathbf{t}}

\newcommand{\ZZ}{\mathbb{Z}}
\newcommand{\QQ}{\mathbb{Q}}
\newcommand{\RR}{\mathbb{R}}
\newcommand{\CC}{\mathbb{C}}
\newcommand{\PP}{\mathbb{P}}

\newcommand{\Gcon}{G_{\textrm{con}}}
\newcommand{\Gdis}{G_{\textrm{dis}}}

\definecolor{darkpastelgreen}{rgb}{0.01, 0.75, 0.24}

\DeclareMathOperator{\supp}{\rm supp}

\DeclareMathOperator{\SO}{\rm SO}
\DeclareMathOperator{\SE}{\rm SE}

\newcommand{\mon}{\mathrm{Mon}}
\newcommand{\deck}{\mathrm{Deck}}

\newcommand{\SOCC}{\SO_\CC}

\newcommand{\SECC}{\SE_\CC}
\newcommand{\SERR}{\SE_\RR}

\newcommand{\func}[3]{#1 \colon #2 \rightarrow #3}
\renewcommand{\function}[5]{\begin{align*} #1 \colon #2 &\rightarrow #3 \\ #4 &\mapsto #5 \end{align*}}
\newcommand{\rat}[3]{#1 \colon #2 \dashrightarrow #3}

\newcommand\restr[2]{{
  \left.\kern-\nulldelimiterspace 
  #1 
  \vphantom{\big|} 
  \right|_{#2} 
}}

\newcommand{\xdashrightarrow}[2][]{\ext@arrow 0359\rightarrowfill@@{#1}{#2}}

\newcommand\cycle[2][\;\;]{%
  \readlist\thecycle{#2}%
  (\foreachitem\i\in\thecycle{\ifnum\icnt=1\else#1\fi\i})%
}

\newcommand{\cam}[2]{[#1\mid #2]}

\newcommand{\wt}[1]{\widetilde{#1}}

\renewcommand{\matrix}[1]{\begin{bmatrix} #1 \end{bmatrix}}

\definecolor{blue(pigment)}{rgb}{0.2, 0.2, 0.8}

\newcommand{\uu}{\mathbf{u}}

\newcommand{\xdasharrow}[2][->]{
\tikz[baseline=-\the\dimexpr\fontdimen22\textfont2\relax]{
\node[anchor=south,font=\scriptsize, inner ysep=1.5pt,outer xsep=2.2pt](x){#2};
\draw[shorten <=3.4pt,shorten >=3.4pt,dashed,#1](x.south west)--(x.south east);
}
}

\newcommand{\ceil}[1]{\left\lceil #1 \right\rceil}

\crefname{algocf}{alg.}{algs.}
\Crefname{algocf}{Algorithm}{Algorithms}

\newcommand{\str}{\,{\odot}\,\,}

\DeclarePairedDelimiter{\norm}{\lVert}{\rVert}
\DeclarePairedDelimiter{\normsq}{\lVert}{\rVert^2}

\journal{Journal of Symbolic Computation}

\begin{document}

\begin{frontmatter}




\title{Using monodromy to recover symmetries of polynomial systems}

\author[label1]{Timothy Duff}
\affiliation[label1]{
  organization={University of Washington},
  city={Seattle},
  state={Washington},
  country ={USA},
  postcode = {98195-4350}
}

\author[label2]{Viktor Korotynskiy}
\affiliation[label2]{
  institution={CIIRC, CTU},
  city={Prague},
  country={Czech Republic}
}
\author[label2]{Tomas Pajdla}
\author[label3]{Margaret H. Regan}
\affiliation[label3]{
  institution={College of the Holy Cross},
  city={Worcester},
  state={Massachusetts},
  country={USA},
  postcode={01610-2395}
}

\begin{abstract}
Galois/monodromy groups attached to parametric systems of polynomial equations provide a method for detecting the existence of symmetries in solution sets.
Beyond the question of existence, one would like to compute formulas for these symmetries, towards the eventual goal of solving the systems more efficiently.
We describe and implement one possible approach to this task using numerical homotopy continuation and multivariate rational function interpolation.
We describe additional methods that detect and exploit \emph{a priori} unknown quasi-homogeneous structure in symmetries.
These methods extend the range of interpolation to larger examples, 
including applications with nonlinear symmetries drawn from vision and robotics.
\end{abstract}







\end{frontmatter}


\section{Introduction}\label{sec:intro}

Structured systems of nonlinear equations appear frequently in applications like computer vision and robotics.
Although the word ``structure'' can be interpreted in many ways, one of its aspects that is strongly connected to the complexity of solving is the \emph{algebraic degree} of the problem to be solved.
In many contexts, this may simply refer to the number of solutions of a system (usually counted over the complex numbers).
However, if we adopt this definition without scrutiny, we may fail in certain special cases to detect additional structure such as symmetry.


To answer more refined questions involving structure, one can often consider a \emph{Galois/monodromy} group naturally associated to the problem of interest.
In this case, ``problem'' refers to a \emph{parametric family} of problem instances which must be solved for different sets of parameter values.
In our work, we are primarily interested in \emph{geometric} Galois groups arising from algebraic extensions of functions fields of varieties over the complex numbers.\footnote{See eg.~\cite[\S 1.2]{galois-survey} for a discussion of how other fields of definition relate to this setup.}

Currently, a number of heuristic methods for computing Galois/monodromy groups using numerical homotopy continuation methods have been proposed and implemented, eg.~\cite{NumGalois,duff-monodromy}.
It is also fairly well-understood how Galois/monodromy groups encode important structural properties such as \emph{decomposability}, or the existence of problem symmetries which may be expressed as rational functions known as \emph{deck transformations.}
Thus, Galois/monodromy group computation provides us a useful toolkit for detecting the \emph{existence} of special structure.
However, one key challenge remains: once we know that our problem \emph{does} have such special structure, can we use this information to solve systems more efficiently?

Our work focuses on a natural first step towards addressing this challenge: \emph{given the data of a numerical Galois/monodromy group computation, can we recover formulas for the deck transformations?}

In this paper, we describe and implement a novel framework that solves this problem, namely recovering symmetries.
Our framework combines  techniques for numerically computing Galois/monodromy groups with a new scheme for floating-point interpolation of multivariate rational functions which represent the underlying symmetries. 
This leads to an improvement upon the prior interpolation techniques described in~\cite{us_ISSAC23}.  
While the novelty of~\cite{us_ISSAC23} was the combination of monodromy techniques with interpolation, this paper provides a substantial new contribution in the form of expanding the interpolation methods to detect a quasi-homogeneous structure and exploit this when interpolating (see Section~\ref{sec:symmetries}).  As a result, previously intractable problems are now solvable.  For example, improvements are made on the five-point relative pose problem from~\cite{us_ISSAC23} (see Section~\ref{subsec:five-point-problem}). 
New problems, such as the Perspective-3-Point (P3P) and radial camera relative pose, can also now be solved, as described in Section~\ref{subsec:p3p} and Section~\ref{subsec:radial}, respectively.




In~\Cref{sec:previous-works}, we provide some context for our approach by considering related previous works.
In~\Cref{sec:background}, we establish terminology and useful background facts.
In~\Cref{sec:basic-method}, we describe our a basic approach to interpolating deck transformations and illustrate it on simple examples. 
In~\Cref{sec:symmetries}, we describe a more sophisticated variant of this approach, which exploits the quasi-homogeneous structure of certain symmetries to reduce the size of the associated linear algebra subproblems. 
In~\Cref{sec:ex}, we describe implementation of our accompanying software package \texttt{DecomposingPolynomialSystems} for the Julia programming language~\cite{julia}, along with experiments on example problems drawn from engineering.
The source code for this package may be obtained at the url below:
\textcolor{magenta}{\url{https://github.com/MultivariatePolynomialSystems}}.

\section{Related work}\label{sec:previous-works}

Galois/monodromy groups have long had a presence in algebraic computation, used as a tool in the study of algebraic curves, polynomial factorization, and numerical irreducible decomposition~\cite{GalligoPoteaux,VanHoeij,NID}.
In recent years, monodromy-based methods have become a popular heuristic for computing the isolated solutions of parametric polynomial systems~\cite{duff-monodromy, Critical}.
One appealing aspect of these methods is that they are useful for constructing efficient \emph{start systems} to be used in parameter homotopies, particularly in cases where more traditional start systems (total degree, polyhedral) fail to capture the full structure.
Another appealing feature is that symmetry or decomposability can be naturally incorporated in both the offline monodromy and online parameter homotopy phases.
This is the main idea behind several recent, closely-related works which use \emph{a priori} knowledge of symmetries to speed up solving~\cite{Amendola,yahl}.
In contrast to these works, our approach recovers symmetries with no such knowledge, and with limited assumptions on the system to be solved.
Our work is also a natural continuation of the paper~\cite{GaloisComputerVision}, where Galois/monodromy groups were used to infer decompositions and symmetries that were not previously known for some novel problems in computer vision.
Here, our emphasis is a novel \emph{method}, illustrated on a variety of examples.

Interpolation is a well-studied problem in symbolic-numeric computation and an important ingredient for solving our recovery problem.
In our work, we are faced with the difficult task of interpolating an \emph{exact} rational function (as opposed to some low-degree approximation) from \emph{inexact inputs} in double-precision floating-point arithmetic.
For this reason, we employ many heuristics, and make no attempt to match state-of-the-art interpolation techniques.
On the other hand, we hope that experts on interpolation will view our particular application as a potential use case for their own methods.
Some relevant references for the specific problem of multivariate rational function interpolation include~\cite{DBLP:conf/issac/KaltofenY07,DBLP:journals/tcs/CuytL11,DBLP:journals/cca/HoevenL21}.

Our focus on inexact inputs is due to the fact that interpolation occurs downstream of numerical homotopy continuation in our framework.
This is also why we cannot pick inputs for the interpolation problem arbitrarily.
With that said, we point out that assuming exact inputs could also be relevant if, say, certified homotopy continuation (see~\cite{xu-burr-yap,berltran-leykin,hauenstein-liddell-haywood,vdH:homotopy}) is used, augmented by some additional postprocessing.


Additional methodology introduced in this version of the paper relies on detecting structure that was \emph{a priori} unknown quasi-homogeneous structure. 
These techniques rely on discovering scaling symmetries by computing the Smith Normal Form of an integer matrix, a well-studied computational problem. Detecting such scaling symmetries with the integer linear algebra techniques has been done before in multiple different contexts~\cite{yahl2,Larssonsymm,Corless,Hubert1,Hubert2}. 
The advances in this paper highlight the use of the Smith Normal Form for both the variables and parameters, instead of simply the variables.

\section{Background}\label{sec:background}

In this work, we are interested in solving polynomial systems whose solutions correspond to points in a generic fiber of a branched cover of complex algebraic varieties.
Here we collect some definitions and theoretical facts that we need to work within this framework.
The section concludes with~\Cref{prop:deck_via_paths} and~\Cref{cor:deck_after_tracking}, which justify the correctness of our general interpolation setup used in~\Cref{sec:basic-method,sec:symmetries}.

\begin{defn}\label{def:branched-cover}
Let $X$ and $Z$ be irreducible algebraic varieties of dimension $m$ over the complex numbers.
A \emph{branched cover} is a dominant, rational map $f:X \dashrightarrow Z$.
The varieties $X$ and $Z$ are called the total space and the base space of the cover, respectively.
The number of (reduced) points in the preimage over a generic $z \in Z$ is called the degree of $f,$ denoted $\deg(f).$
\end{defn}

Essentially, the base space $Z$ in~\Cref{def:branched-cover} can be thought of as a space of parameters or observations.
The fiber $f^{-1} ( z)$ over some particular $z\in Z$ should usually be understood as the solutions of a particular problem instance specified by $z.$
Oftentimes, $Z $ may be assumed to be an affine space $\CC^m$, and in this case we write $\pp \in \CC^m$ for parameter values.
The assumptions that $f$ is dominant and $\dim X = m$ imply that there is a finite, nonzero number of solutions for almost all parameters.
Counting solutions over $\CC, $ that number is $\deg(f).$
Additionally, the total space $X$ is often either
\begin{enumerate}
\item an irreducible variety consisting of \emph{problem-solution pairs}, 
\begin{equation}\label{eq:ps-pair-space}
X = \{ (\xx, \pp) \in \CC^{n+m} \mid f_1 (\xx , \pp ) = \cdots = f_k (\xx ,\pp ) = 0 \}
\end{equation}
for some system of polynomials $f_1, \ldots , f_k \in \CC [\xx , \pp ]$, with projection $f: X \to \CC^m$ given by $f(\xx , \pp ) = \pp ,$ or
\item an affine space of unknowns $X = \CC^m$, and $f: \CC^m \dashrightarrow \CC^m $.
\end{enumerate}

Cases (1) and (2) for the total space $X$ given above are closely related.
Indeed, (2) reduces to (1) if we take $X$ to be the graph of $f.$
Conversely, it can often be the case that the variety $X$ has a unirational parametrization $p: \CC^m \dashrightarrow X.$
In this case, (1) reduces to (2) by replacing $f$ with the branched cover $f \circ p : \CC^m \dashrightarrow \CC^m .$
When $\deg(f)$ and $\deg(p)$ are both greater than $1,$ the composite map $f \circ p$ is an example of a decomposable branched cover.

\begin{defn}\label{def:decomposable-branched-cover}
A branched cover $f: X \dashrightarrow Z$ is said to be \emph{decomposable} if there exist two branched covers $g : X \dashrightarrow Y$ and $h : Y \dashrightarrow Z$  with $\deg(g), \deg(h) < \deg(f)$ such that $f (x) = h \circ  g (x)$ for all $x$ in a nonempty Zariski-open subset of $X.$
The maps $g$ and $h$ give a \emph{decomposition} of $f.$
\end{defn}

\begin{example}\label{ex:palindromic-sextic}
Let $X = \VV ( a x ^6 + b x^5 + c x^4  + d x^3 + c x^2 + b x + a) \subset \CC^5$, $Z= \CC^4,$ and $f : X \to Z$ given by $f(a,b,c,d,x) = (a,b,c,d).$
The projection $f$ is a decomposable branched cover in the sense of~\Cref{def:decomposable-branched-cover}.
To see this, take $Y = \VV (a (y^3 - 3y) + b (y^2 - 2) + c y + d) \subset \CC^5,$ and define $g: X \dashrightarrow Y$ by $g(a,b,c,d,x) = (a,b,c,d, \frac{x^2+1}{x}),$ and $h : Y \rightarrow Z$ by $h(a,b,c,d,y) = (a,b,c,d).$
The degrees of maps satisfy $6 = \deg(f) = \deg (h \circ g) = \deg (h) \cdot \deg (g) = 3 \cdot 2.$
\end{example}

\begin{example}\label{ex:sparse-triangular-system}
The following example is based on~\cite[\S 2.3.2]{yahl}, and belongs to a general class of examples where decomposability can be detected via equations' Newton polytopes. 
Let $Z = \CC^{23},$ and $X\subset  \CC^{26}$ be the vanishing locus of the three equations below:
\begin{align*}
a\,x^{3}y\,z^{4}+b\,x^{2}y^{2}z^{4}+c\,x^{2}y\,z^{3}+d\,x\,y^{2}z^{3}+e\,x^{2}z^{2}+f\,x\,y\,z^{2}+g\,x\,z+h,\\ 
i\,x^{3}y\,z^{4}+j\,x^{2}y^{2}z^{4}+k\,x^{2}y\,z^{3}+l\,x\,y^{2}z^{3}+m\,x^{2}z^{2}+n\,x\,y\,z^{2}+o\,x\,z+p,\\
q\,x\,y\,z^{4}+r\,y\,z^{5}+s\,x\,z^{3}+t\,z^{4}+u\,z^{3}+v\,z^{2}+w.
\end{align*}
The projection $f: X \rightarrow \CC^{23}$ given by $f(a, \ldots , z) \mapsto (a, \ldots , w)$ is a branched cover of degree $32.$
If we let $Y$ be the set of all $\left(a,\ldots , w, \hat{x}, \hat{y}\right)\in   \CC^{25}$ such that
\begin{align*}
a\,\hat{x}^{3}\hat{y}+b\,\hat{x}^{2}\hat{y}^{2}+c\,\hat{x}^{2}\hat{y}+d\,\hat{x}\,\hat{y}^{2}+e\,\hat{x}^{2}+f\,\hat{x}\,\hat{y}+g\,\hat{x}+h &=\\ i\,\hat{x}^{3}\hat{y}+j\,\hat{x}^{2}\hat{y}^{2}+k\,\hat{x}^{2}\hat{y}+l\,\hat{x}\,\hat{y}^{2}+m\,\hat{x}^{2}+n\,\hat{x}\,\hat{y}+o\,\hat{x}+p &= 0,
\end{align*}
then $g : X \to Y$ given by $g(a, \ldots , w, x,y,z)  = (a, \ldots , w, xz, yz)$ and $h : Y \to Z$ given by $h(a, \ldots , w, \hat{x}, \hat{y}) = (a, \ldots , w)$ show that $f$ is decomposable in the sense of~\Cref{def:decomposable-branched-cover}.
Here we have $\deg (h) = 8$ and $\deg (g) = 4.$
\end{example}

The Galois/monodromy group is an invariant that allows us to decide whether or not a branched cover is decomposable, without actually exhibiting a decomposition.
We recall the basic definitions here.
For a branched cover $f:X \dashrightarrow Z,$ fix a dense Zariski-open subset $U \subset Z$ such that $f^{-1} (z)$ consists of $d=\deg (f)$ points.
Over a regular locus, the branched cover $f$ restricts to a $d$-sheeted covering map in the usual sense given by $f^{-1} (U) \to U.$
For any basepoint $z\in U,$ we may construct via \emph{path-lifting} a group homomorphism from the fundamental group $\pi_1 (U ; z)$ to the symmetric group $S_d$. 

More precisely, if $\gamma : [0,1] \to U$ is any map that is continuous with resepct to the Euclidean topology, then the \emph{unique lifting property}~\cite[Prop.~1.34]{hatcher} implies that there are precisely $d$ continuous lifts $\tilde{\gamma}_1, \ldots , \tilde{\gamma}_d : [0,1] \to \pi^{-1} (U)$ satisfying $f\circ \tilde{\gamma}_i (t) = \gamma_i (t)$ for all $i=1, \ldots , d$ and $t\in [0,1].$ 
In particular, $\tilde{\gamma}_i (0), \tilde{\gamma}_i (1) \in f^{-1} (z)$, and there is a permutation $\sigma_{\gamma }$ that sends each $\tilde{\gamma}_i(1)$ to $\tilde{\gamma}_i(0).$
One may check that this permutation is independent of the chosen representative $\gamma $ of the homotopy class $[\gamma ]\in \pi_1 (U; z).$
Thus, for our chosen $U$ and $z$ we may define the \emph{monodromy representation},
\begin{align}
\rho_{u,Z} : \pi_1 (U ; z) &\to S_d \label{eq:monodromy-rep}\\
[\gamma ] &\mapsto \sigma_\gamma . \nonumber
\end{align}
This gives a group homomorphism, whose image is a subgroup of $S_d,$ which turns out to be independent of the choice of $U$ and $z.$ 

\begin{defn}\label{def:gm-group}
The \emph{Galois/monodromy group} of a branched cover $f$ is
the subgroup of $S_d$ given by the image of the map~\eqref{eq:monodromy-rep}.
\end{defn}

The abstract structure of the Galois/monodromy group, although interesting, is not our main focus.
Instead, we will be mainly interested in the action of this group given by~\eqref{eq:monodromy-rep}.
Since $X$ is irreducible, this action is transitive (see eg.~\cite[Lemma 4.4, p87]{Miranda}).

The monodromy action also provides a clean characterization of decomposable branched covers.
Recall that the action of a group $G$ on a finite set $B$ is said to be \emph{imprimitive} if there exists a nontrivial partition $B = B_1 \sqcup B_2 \sqcup \cdots \sqcup B_k$ such that for any $g\in G$ and $B_i$ there exists a $B_j $ with $g\cdot B_i = B_j.$
If $B$ has $d$ elements and $G$ is a finite, transitive subgroup of $S_d$, it follows that the subsets $B_i$ must all have the same size. 
The sets $B_1, \ldots ,B_k$ are called \emph{blocks} of the imprimitive action, and are said to form a \emph{block system}.

\begin{proposition}\label{prop:decomposability}
(See eg.~\cite[Proposition 1]{yahl}.)
A branched cover is decomposable if and only if its Galois/monodromy group is imprimitive.
\end{proposition}

\begin{example}\label{ex:decomposable}
For the branched cover $f$ from~\Cref{ex:palindromic-sextic}, the Galois/monodromy group acts transitively on the set of roots, which we replace with a set of labels $B = \{1 , \ldots , 6 \}$. 
Up to relabeling, there is a block decomposition for this action given by $ B = \{ 1, 2, 3 \} \sqcup \{ 4, 5, 6 \}.$
There are $48 = 2^3 \cdot 3! $ permutations in $S_6$ that preserve this block decomposition.
These permutations form a group called the \emph{wreath product} $S_2 \wr S_3.$
This group can be presented by three permutation generators, for instance
\begin{equation}\label{eq:wr-2-3-presentation}
\langle 
(1 2) (4 5), (1 2 3) (4 5 6), (1 4) (2 5) (3 6)
\rangle .
\end{equation}
Computing the Galois/group monodromy group numerically, we find that every element of $S_2 \wr S_3$ arises as $\sigma_\gamma $ for some loop $\gamma .$

Similarly, for the branched cover from~\Cref{ex:sparse-triangular-system}, we find by numerical computation that its Galois/monodromy group is the wreath product $S_4 \wr S_8$, a group of order $(4!)^8 \cdot 8!$
\end{example}

In general, a transitive, imprimitive permutation group has a block system $B_1, \ldots, B_k$ whose blocks all have the same size $l,$ and is thus permutation-isomorphic to a subgroup of the wreath product $S_l \wr S_k.$
Unlike the previous example, there are a number of surprising cases of decomposable branched covers where the Galois/monodromy group is a \emph{proper} subgroup of the associated wreath product: for instance, the five-point problem of~\Cref{subsec:five-point-problem}.

We point out that~\Cref{prop:decomposability} dates back, at least in some form, to work of Ritt on polynomial decompositions~\cite{MR1501205}.
This work is directly related to \emph{decomposition problems} for polynomials and rational functions studied in computer algebra (see eg.~\cite{DBLP:conf/issac/FaugereGP10, DBLP:conf/casc/Gathen0R99}).

However, the main focus in this paper is not decomposability \emph{per se}.
Rather, we are interested in a property that is usually stronger: the existence of symmetries.
A natural, and general, notion of symmetry can be obtained by studying the embedding of function fields $f^* :\CC (Z) \dashrightarrow \CC (X)$ induced by a branched cover.
The field extension $\CC (X) / \CC (Z)$, although not usually a Galois extension, may nevertheless a have a nontrivial group of automorphisms.
These automorphisms correspond to rational maps $\Psi : X \dashrightarrow X $ with $f \circ \Psi = f.$
Topologically, these comprise the group of \emph{deck transformations} of $f.$

\Cref{prop:centralizer} below explains the relationship between deck transformations and decomposability, and provides an analogue of~\Cref{prop:decomposability} for detecting the existence of deck transformations.
Proofs are in~\cite[\S 2.1]{GaloisComputerVision}.

\begin{proposition} \label{prop:centralizer}
Let $f :X \dashrightarrow Z$ be a branched cover of degree $d.$
\begin{enumerate}
\item If $f$ has a nontrivial deck transformation group, then its Galois/monodromy group is decomposable or cylic of order $d.$ 
(Both hold for composite $d$.)
\item Restricting the deck transformations to the fiber $f^{-1} (z)$ defines another permutation group which is the centralizer of the Galois/monodromy group in $S_d.$
In particular, there exists a nontrivial deck transformation if and only if this centralizer is nontrivial. 
\end{enumerate}
\end{proposition}

\begin{example}\label{ex:compare-decomp-deck}
For the branched cover $f$ from~\Cref{ex:palindromic-sextic}, the centralizer in $S_6$ of the Galois/monodromy group presented as in~\eqref{eq:wr-2-3-presentation} is a cyclic group of order $2$, namely $\langle  (1 4) (2 5) (3 6) 
\rangle $. 
Correspondingly, there is a nontrivial deck transformation $\Psi : X \dashrightarrow X$ defined by $\Psi (a,b,c,d,x) = (a,b,c,d,1/x).$  

For the branched cover $f$ of~\Cref{ex:sparse-triangular-system}, the centralizer of its Galois/monodromy group $S_4 \wr S_8$ in $S_{32}$ is trivial.
Thus, this decomposable branched cover has no nontrivial deck transformations.
\end{example}

In the final results of this section, \Cref{prop:deck_via_paths} and \Cref{cor:deck_after_tracking}, we use the terminology \emph{generic path} for a given branched cover $f: X \dashrightarrow Z.$ 
This means a path $\alpha : [0,1] \to U$ where $U$ is some suitably small set, either a regular locus in $Z$ or its preimage in $X.$
In the former case, we write $\wt{\alpha}_x$ for the unique lift of a path $\alpha$ through $f$ starting at $x \in f^{-1} (\alpha (0))$.

\begin{figure}[h!]
\centering
\includegraphics[scale=0.75]{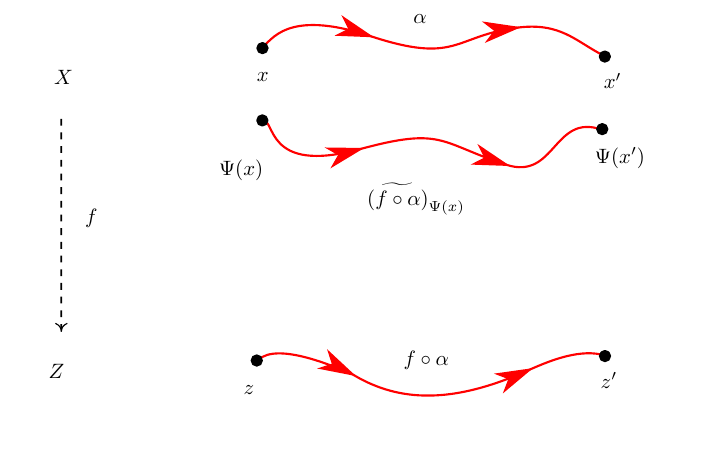}
\caption{Illustration of~\Cref{prop:deck_via_paths}.}
\label{fig:deck_via_paths}
\end{figure}

\begin{proposition} \label{prop:deck_via_paths}
Let $\rat{f}{X}{Z}$ be a branched cover with a fixed generic point $x \in X$. Then the value of a deck transformation $\Psi \in \deck(X/Z)$ at a generic point $x' \in X$ is completely determined via path-lifting by where it sends $x$. 
Explicitly, 
\begin{equation}  \label{eq:deck_via_paths}
    \Psi(x') = \wt{\left(f \circ \alpha \right)}_{\Psi(x)}(1),
\end{equation}
where $\alpha$ is a generic path in $X$ from $x$ to $x'$ (see Figure \ref{fig:deck_via_paths}). 
\end{proposition}
\begin{proof}
We refer to the proof of~\cite[Prop.~1.33]{hatcher} and the general definition of a \emph{lift} given on~\cite[p.~60]{hatcher}.
The deck transformation $\Psi$ is a lift of $f$ to $X$ in the sense of this definition.
This means the proof of Proposition 1.33 can be applied to construct a deck transformation $\Psi '$ with $\Psi ' (x) = \Psi (x).$
This construction uses lifts of a generic path $\alpha $ to construct $\Psi '$, with the additional property that $\Psi ' (x') = \wt{\left(f \circ \alpha \right)}_{\Psi(x)}(1).$
The unique path-lifting property then implies that $\Psi (x') = \Psi' (x')$.
\end{proof}

A consequence of~\Cref{prop:deck_via_paths} is that the correspondence between solutions for a fixed set of parameters under a fixed deck transformation $\Psi $ is preserved under path-lifting.

\begin{figure}[h!]
\centering
\includegraphics[scale=0.75]{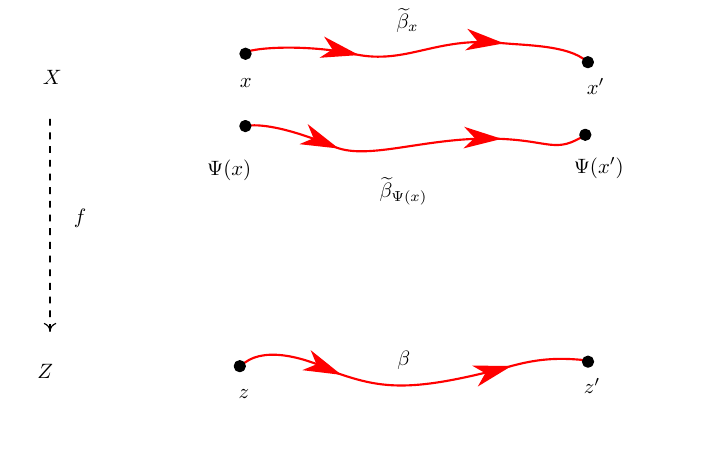}
\caption{Illustration of~\Cref{cor:deck_after_tracking}.}
\label{fig:deck_after_tracking}
\end{figure}

\begin{cor} \label{cor:deck_after_tracking}
Let $\rat{f}{X}{Z}$ be a branched cover and $\Psi \in \deck(X/Z)$. Let $z \in Z$ be a generic point and $\beta$ be a generic path in $Z$ starting at $z$ (see Figure \ref{fig:deck_after_tracking}). Then for $x \in X_{z}$ we have
\[ \Psi(\wt{\beta}_x(1)) = \wt{\beta}_{\Psi(x)}(1) \]
In other words, the points in two lifts of $\beta$---one starting at $x$, the other starting 
at $\Psi(x)$---are conjugate under $\Psi$ (see~\Cref{fig:deck_after_tracking}).
\end{cor}

\begin{proof}
By \Cref{prop:deck_via_paths}, $\Psi(\wt{\beta}_x(1)) = \wt{\left(f \circ \wt{\beta}_x\right)}_{\Psi(x)}(1)$, which, in turn, is equal to $\wt{\beta}_{\Psi(x)}(1)$, since $f \circ \wt{\beta}_x = \beta$.
\end{proof}

\section{Basic method -- dense interpolation}\label{sec:basic-method}

Consider a branched cover encoding problem-solution pairs $(\xx , \pp )$, 
\begin{align}
f : X &\to \CC^m \label{eq:bc-affine} \\
(\xx , \pp ) &\mapsto \pp \nonumber 
\end{align}
with $\deg (f) = d$, which has a nontrivial deck transformation
\begin{equation}\label{eq:deck}
\Psi(\xx,\pp) = \matrix{\psi_1(\xx,\pp) & \dots & \psi_n(\xx,\pp) & \pp^\top}^\top.
\end{equation}

As mentioned in the introduction, we may compute the Galois/monodromy group of $f$ using numerical homotopy continuation.
This is possible provided that we make the following assumptions about how our branched cover is given as input.

\begin{assumption}\label{assumption:sampling-and-equations}
For the branched cover defined in~\eqref{eq:bc-affine}, assume that $n$ rational functions $f_1, \ldots , f_n$ vanishing on $X$ are known, and that we have access to a sampling oracle that produces generic $(\xx^* , \pp^*) \in X$ such that the $n\times n$ Jacobian $\frac{\partial \mathbf{f}}{\partial \xx} (\xx^* , \pp^*) $ has rank $n.$
\end{assumption}

\Cref{assumption:sampling-and-equations} is often satisfied in practice, including cases where even a set-theoretic description of $X$ is not known.
Additionally, we assume that homotopy continuation---specifically, coefficient parameter homotopy---can be used to track $d$ known solutions for fixed, generic parameter values $\pp^*$ (corresponding to $f^{-1} (\pp^*)$) to $d$ solutions for some other parameter values $\pp\in \CC^m$ (corresponding to $f^{-1} (\pp)$).
These parameter homotopies are the basis of the unspecified subroutines in lines 1 and 9 of~\Cref{alg:deck-interpolation}.

An important observation is that we can interpolate each of the coordinate functions $\psi_j(\xx,\pp)$ in~\eqref{eq:deck} independently. We assume that the rational function $\psi_j$ contains only monomials up to total degree $D$. 
Since these monomials may or may not involve the parameters $\pp$, we distinguish the \emph{parameter-dependent} and \emph{parameter-independent} settings, in which we take the number of monomials $t$ to be either
\begin{align}
t &= {n+m+D \choose D}, \text{ or}   \tag{for parameter-dependent $\psi_j (\xx, \pp)$}\\
t &= {n+D \choose D}    \tag{for parameter-independent $\psi_j (\xx )$} \label{eq:t-param-indep}.
\end{align}
Our task is then to recover two vectors of unknown coefficients 
\[
\aa = \matrix{a_1 & \dots & a_t}^\top , \, \, \bb = \matrix{b_1 & \dots & b_t}^\top \in \CC^t,\]
such that $\psi_j $ can be represented on $X$ as
\begin{equation}\label{eq:psi-ab} 
\psi_{\aa , \bb} (\xx,\pp) = \frac{\sum_{k=1}^{t} a_k\cdot(\xx,\pp)^{\al_k}}{\sum_{k=1}^{t} b_k\cdot(\xx,\pp)^{\be_k}} .\end{equation}
In~\Cref{eq:psi-ab}, the vectors $\al_k, \be_k \in \ZZ_{\ge 0}^{n+m}$ range over a suitable set of multidegrees, depending on whether we are in the parameter-dependent or parameter-independent setting. 
If we know that $(\xx_i', \pp_i) = \Psi(\xx_i, \pp_i)$ for points $(\xx_i , \pp_i), (\xx_i ' , \pp_i ) \in X$, this gives a homogeneous linear constraint on $\aa$ and $\bb$,
\begin{equation} \label{eq:constraint}
    \sum_{k=1}^t a_k\cdot(\xx_i,\pp_i)^{\al_k} - x'_{i j}\cdot\left(\sum_{k=1}^t b_k\cdot(\xx_i,\pp_i)^{\be_k}\right) = 0.
\end{equation}
Suppose we have already computed permutations generating the monodromy group based at parameter values $\pp_1 \in \CC^m$, and let $\xx_1, \xx_1'$ be two solutions with $\Psi (\xx_1 , \pp_1) = (\xx_1 ' , \pp_1)$.
\Cref{prop:centralizer} implies that $\sigma \cdot (\xx_1 , \pp_1) = (\xx_1 ', \pp_1 ')$ for some element of the centralizer $\sigma \in \mathrm{Cent}_{S_d}(\mathrm{Mon}(f, \pp_1))$ corresponding to $\Psi .$
Now,~\Cref{cor:deck_after_tracking} implies that we may obtain additional sample points satisfying~\eqref{eq:constraint} by tracking parameter homotopies using the system $f_1, \ldots , f_n.$
Specifically, we may track the solution curves with initial values $\xx_1, \xx_1'$ from $\pp_1$ to generic $\pp_i \in \CC^m$ for $i=1,\ldots , 2t$, which then allows us recover the coordinate functions of $\Psi .$

\begin{proposition}[Correctness of~\Cref{alg:deck-interpolation}]\label{prop:correctness}
Suppose that $\psi_j$ in~\eqref{eq:deck} can be represented as the quotient of polynomials with degree $\le D$ and $t$ monomials each.
For a sufficiently generic sample 
\[
(\xx_1, \pp_1), \ldots , (\xx_{2t}, \pp_{2t}) , (\xx_1', \pp_1), \ldots , (\xx_{2t}', \pp_{2t})\in X,
\]
with $(\xx_i', \pp_i)= \Psi (\xx_i, \pp_i) $ for all $i$,
suppose $\matrix{\aa^\top & \bb^\top}$ is a solution to the $2t$ linear equations given by~\eqref{eq:constraint} for $i=1,\ldots , 2t$, 
which lies outside the span of all solutions with $\aa = \mathbf{0}$ or $\bb = \mathbf{0}.$ 
Then the rational function obtained by restricting $\psi_{\aa,\bb}(\xx,\pp)$ to $X$ equals $\psi_j$. 
\end{proposition}

\begin{proof}
The assumption that $\matrix{\aa^\top & \bb^\top}$ is a nontrivial linear combination of solutions with $\aa , \bb \ne \mathbf{0}$ ensures that $\psi_{\aa , \bb}$ is a well-defined, nonzero rational function on $X.$
Such a function of the form~\eqref{eq:psi-ab} is determined by its values on $2t$ generic points of $X.$
Since $\psi_j$, by assumption, is also such a function, the $2t$ linear constraints~\eqref{eq:constraint} force $\psi_j$ and $\psi_{\aa , \bb}$ to agree on $X.$
\end{proof}
Thus, to interpolate $\psi_j$, we may  
determine from the linear equations~\eqref{eq:constraint} a $2t\times2t$ Vandermonde-type coefficient matrix $\mathbf{A}$.
We represent the nullspace of $\mathbf{A}$ by the column-span of a matrix $\mathbf{N}$ with $2t$ rows. 
Although~\Cref{prop:correctness} can be viewed as a uniqueness statement, the matrix $\mathbf{N}$ will generally have more than one column, even for generic samples $(\xx_1, \pp_1 ), \ldots , (\xx_{2t} , \pp_{2t}) \in X.$
The ``extra" columns of $\mathbf{N}$ appear for two reasons:
\begin{enumerate}
    \item There may exist different representatives of $\psi_j$ on $X$ of the form~\eqref{eq:psi-ab}, whose coefficient vectors are linearly independent.
    \item The nullspace of $\mathbf{A}$ may contain \emph{spurious solutions} not satisfying the hypothesis $\aa = \mathbf{0}$ or $\bb = \mathbf{0}$ in~\Cref{prop:correctness}. For instance, fixing $\bb = \mathbf{0}$ we may interpolate polynomial functions of the form $\sum_{k=1}^t a_k\cdot(\xx,\pp)^{\al_k}$ vanishing on $X$. In the same way, fixing $\aa = \mathbf{0}$ we interpolate polynomial functions of the form $\sum_{k=1}^t b_k\cdot(\xx,\pp)^{\be_k}$ vanishing on $X$. 
\end{enumerate}
For some applications it may be necessary to pick a sparse representative from the nullspace of $\mathbf{A}$.
In general, finding the sparsest vector in the nullspace of a matrix is NP-hard~\cite{nullspace-hard}.
Nevertheless, in many cases we may find a relatively good sparse representative by looking at the reduced row echelon form of $\mathbf{N}^\top$ for some particular ordering of its columns and picking one with the fewest zeros subject to the additional constraints $\aa, \bb \neq \mathbf{0}$.
We illustrate some of the choices involved on two simple examples.

\begin{example} \label{ex:1}
Let $X = \VV (x^2 + px + 1 )$, $f (x,p) = p.$ The Galois/monodromy group and deck transformation group are both $S_2.$
When interpolating a nontrivial deck transformations of degree $D = 1$, we obtain the reduced row echelon form for 
$\mathbf{N}^\top$ below.
\[
\wt{\mathbf{N}} =  
\begin{blockarray}{cccrccc}
1 & x & p & 1 & x & p \\
\begin{block}{[cccrcc]l}
  \bigstrut[t] 1 & 0 & 0 & 0 & 1 & 0 & \frac{1}{x} \\
  0 & 1 & 1 & -1 & 0 & 0 & -x-p \bigstrut[b] \\
\end{block}
\end{blockarray}
 \]
We see that $\Psi(x,p)$ has 2 different representatives $\frac{1}{x}$ and $-x-p$, which both agree on $X.$ 
There is no clear choice of ``best representative''.
In terms of sparsity, the representative $\frac{1}{x}$ is superior. However, one might instead prefer $-x-p$ since it is a polynomial.
\end{example}

\begin{example} \label{ex:2}
Consider the branched cover 
\function{f}{\VV (x^2 + x + p, x + y + p)}{\CC}{ (x,y,p)}{p,}
which has a unique non-identity deck transformation $\Psi = (\psi_1, \psi_2).$
If we interpolate parameter-dependent deck transformations, we may find matrices $\mathbf{A}_1$ and $\mathbf{A}_2$ representing $\Psi$ which are $8\times 8$.
The reduced row echelon forms of the transposed nullspaces are
\[
\wt{\mathbf{N}_1} = 
\begin{blockarray}{rrrrrrrrr}
1 & x & y & p & 1 & x & y & p \\
\begin{block}{[rrrrrrrr]l}
  \bigstrut[t] 1 & 0 & -1 & 0 & -1 & 0 & -1 & -1 & \frac{1-y}{-1-y-p} \\ [5pt]
  0 & 1 & 1 & 0 & 0 & 0 & 1 & 1 & \frac{x+y}{y+p} \\ [5pt]
  0 & 0 & 0 & 1 & 0 & 0 & -1 & -1 & \frac{p}{-y-p} \\ [5pt]
  0 & 0 & 0 & 0 & 0 & 1 & 1 & 1 & \text{spurious,}
 \bigstrut[b] \\
\end{block}
\end{blockarray}
\]
and for $\psi_2(x,y,p)$ we have
\[
\wt{\mathbf{N}_2} = 
\begin{blockarray}{rrrrrrrrr}
1 & x & y & p & 1 & x & y & p \\
\begin{block}{[rrrrrrrr]l}
  \bigstrut[t] 1 & 0 & -1 & -2 & 1 & 0 & 0 & 0 & 1-y-2p \\ [5pt]
  0 & 1 & 1 & 1 & 0 & 0 & 0 & 0 & \text{spurious} \\ [5pt]
  0 & 0 & 0 & 0 & 0 & 1 & 1 & 1 & \text{spurious.}
 \bigstrut[b] \\
\end{block}
\end{blockarray}
\]

If we are not interested in the sparsest representative, then we may take $\psi_1 = \frac{p}{-y-p}$ and $\psi_2 = 1-y-2p$.

In this example, it is possible to find the sparsest polynomial representative for $\psi_1$ by solving an auxiliary linear system. 
In other words, we compute a linear combination of rows $\matrix{\aa^\top & \bb^\top} = \rr^\top \wt{\mathbf{N}_1}$ such that $\bb^\top = \matrix{1 & \mathbf{0}^\top}$ and $\aa^\top $ contains the minimum number of zeros.
First, to obtain $\bb^\top = \matrix{1 & \mathbf{0}^\top}$, we solve a linear system obtained from the right $4\times 4$ block of $\wt{\mathbf{N}_1}$,
\[ \left(\rr^\top \wt{\mathbf{N}_1} \right)_{:,5:8} = \matrix{1 & \mathbf{0}^\top}. \]
The general solution of this system is given by
\[ \rr^\top = \matrix{-1 & r & r+1 & 0}, \quad r \in \CC. \]
Using $\rr$ to form a linear combination of rows now from the \emph{left} $4\times4$ block of $\wt{\mathbf{N}_1}$, we obtain
\[ \aa^\top = \matrix{-1 & r & r+1 & r+1}. \]
To maximize the sparsity, we may set $r = -1$ to obtain
\[ \matrix{\aa^\top & \bb^\top} = \matrix{-1 & -1 & 0 & 0 & 1 & 0 & 0 & 0} \]
which encodes the function
\[ \psi_1(x,y,p) = -x-1. \]
\end{example}

\begin{algorithm}
\DontPrintSemicolon
  
\KwInput{
$F = (f_1, \ldots , f_n)$ and $(\xx^*, \pp^*)$ as in~\Cref{assumption:sampling-and-equations}, representing $f$ as in~\eqref{eq:bc-affine}; an upper bound for the total degree $D^*$ of monomials in each interpolant
}
  \KwOutput{Partially-specified rational maps representing the group of deck transformations, $\{\Psi_1,\dots,\Psi_q\} = \mathrm{Deck}(f)$, with all coordinate functions representable in degree $\le D^*$ specified}
  $(x^{(1)}, \dots, x^{(d)}), \mathrm{Mon}(f,\pp^*) \gets$ \verb|run_monodromy|($F, \xx^*, \pp^*$)\\
  $\{\sigma_1,\dots,\sigma_q\} \gets \mathrm{Cent}_{S_d}(\mathrm{Mon}(f, \pp^*))$\\
  $\Psi_1 \gets \xx$\\
  \For{$i \gets 2;\ i \leq q;\ i \gets i + 1$}{
    $\Psi_i \gets \matrix{\textbf{missing} & \dots & \textbf{missing}}^\top$
  }
  \For{$D \gets 1;\ D \leq D^*;\ D \gets D + 1$}{
    $t \gets {n+m+D \choose D}$, or ${n+D \choose D}$ if parameter-independent\\
    Track the orbit $\deck(f) \cdot x^{(1)}$ to $2t$ random instances of $F$\\
    \For{$i \gets 2;\ i \leq q;\ i \gets i + 1$} {
      
      \For{$j \gets 1;\ j \leq n;\ j \gets j + 1$} {
      \If{$\Psi_{i_j} \mathrm{\;is \; \mathbf{missing}}$}{
        $\mathbf{A}_j \gets$ $2t \times 2t$ Vandermonde matrix from~\eqref{eq:constraint},\\
        \indent $\quad \xx_k ' = \sigma_i \cdot \xx_k$ for $k=1,\ldots , 2t$ \\
        $\mathbf{N}_j \gets$ nullspace($\mathbf{A}_j$)\\
        $\wt{\mathbf{N}_j} \gets$ rref($\mathbf{N}_j^\top$)\\
        $\matrix{\aa^\top & \bb^\top}^\top \gets$ \texttt{get\char`_representative} ($\wt{\mathbf{N}_j}$)\\
        \If{$\matrix{\aa^\top & \bb^\top}^\top \mathrm{is\; not \; \mathbf{nothing}}$}{
          $\Psi_{i_j} \gets \frac{\sum_{k=1}^{t} a_k\cdot(\xx,\pp)^{\al_k}}{\sum_{k=1}^{t} b_k\cdot(\xx,\pp)^{\be_k}}$
        }
      }
      }
    }
    \If{\text{all} $\Psi_i$ \text{are interpolated}}{
        \Return $\{\Psi_1,\dots,\Psi_q\}$
    }
  }
  \Return $\{\Psi_1,\dots,\Psi_q\}$

\caption{Inhomogeneous interpolation of $\deck (f)$}\label{alg:deck-interpolation}
\end{algorithm}

Our pseudocode in~\Cref{alg:deck-interpolation} outlines a degree-by-degree procedure for interpolating the full set of deck transformations up to a given degree $D^*.$
To implement such a procedure, there are many design choices that could improve performance or meet the needs of a particular task.
Among the design choices, we note that the monodromy, parameter homotopy, and \verb|get_representative| subroutines on respective lines 1, 9, and 17 are left unspecified.
Our implementation relies on \texttt{HomotopyContinuation.jl} for the first two of these subroutines.
For \verb|get_representative|, our implementation chooses the sparsest row in the rref matrix $\wt{\mathbf{N}_j}.$
For the final output of line 19, we heuristically truncate ``small'' entries of $\wt{\mathbf{N}_j}$ of size $<10^{-5}$.

Finally, we note the following improvements to the pseudocode in~\Cref{alg:deck-interpolation}, which we have used in our implementation.
\begin{enumerate}
\item Computing the monodromy group and centralizer in lines 1--2 is an offline task which only needs to be performed once for a given family of systems.
\item In practice, we might only need to recover generators of the deck transformation group.
The needed modifications are trivial, since deck transformations are interpolated independently.
\item To restart the computation at a higher degree limit $D^*,$ one can use previously-computed samples from $X.$
In principle, one can also draw $>2t$ samples and compute the nullspace of the resulting rectangular matrices $\mathbf{A}_j.$
\item To minimize the number of calls to the parameter homotopy subroutine, one can attempt to track samples in ``batches": since every fiber consists of $d$ points and each point gives $1$ constraint on $\psi_j$, then we need to obtain a complete set of $d$ solutions for $r$ different sets of parameters (including $\pp$) such that
\begin{equation}\label{eq:sample-bound}
rd \geq 2t \Rightarrow r \geq \ceil{\frac{2t}{d}}.
\end{equation}
In our experience, this strategy can work well, but comes with the additional caveat that the samples need not satisfy the genericity conditions of~\Cref{prop:correctness}, since multiple parameter values are duplicated.
We encoutered one (ultimately benign) instance of this phenomenon in our study of Alt's problem~\Cref{subsec:alts-problem}. 
In this example, we had $d = 8652 > 2t = 650$, and this strategy resulted in many more spurious rows in $\wt{\mathbf{N}}$ due to all samples using the same parameter values.
\end{enumerate}

\section{Continuous symmetries and multigraded interpolation}\label{sec:symmetries}

As we have already seen, the deck transformations of a branched cover provide us with one useful way to formalize the study of symmetries of a parametric family of polynomial systems.
However, this is by no means the only useful notion of symmetry---deck transformations, although they may depend on both parameters and unknowns, only act nontrivially on the unknowns.
In this section, we consider symmetries that act nontrivially on both parameters and unknowns. 
We assume some basic notions about algebraic groups acting rationally on algebraic varieties, eg. as treated in~\cite[\S 1]{vinberg-popov}.
After a general discussion of branched covers which are equivariant with respect to a general rational action, we specialize to the case of \emph{scaling symmetries}, also known as quasi-torus actions.
The deck transformations of a branched cover equipped with these symmetries are quasi-homogeneous with respect to an associated multigrading.
This leads to~\Cref{alg:graded-deck-interpolation}, a multigraded refinement of~\Cref{alg:deck-interpolation}, which can produce Vandermonde matrices of considerably smaller size.

\subsection{Equivariant Branched Covers}\label{subsec:equivariant}

\begin{defn}\label{def:equivariant}
Let $\rat{f}{X}{Z}$ be a branched cover and $G$ be an algebraic group acting rationally on both $X$ and $Z$. We say $f$ is \emph{$G$-equivariant} if
\begin{equation}\label{eq:equivariant} f(g \cdot x) = g \cdot f(x) \end{equation}
for all $(x,g)$ in some Zariski-open subset of $X \times G$ where both sides of~\eqref{eq:equivariant} are defined. 
If $G$ is irreducible and $\dim G >0,$ we say that the elements of $G$ are \emph{continuous fiber-respecting symmetries of $f$}.
\end{defn}

Since $G$ is a smooth variety, we note that its irreducibility is equivalent to its connectedness in the analytic topology.

The branched covers associated to systems of algebraic equations occurring in applications are typically equivariant with respect to some underlying symmetries of the problem.
For example, the pose estimation problems considered in~\Cref{subsec:five-point-problem,subsec:p3p,subsec:radial} are naturally associated to branched covers which are invariant under certain actions of the \emph{special Euclidean group} $\SECC (3)$, the complexified group of rotations and translations in $3$-space, as well as the scaling symmetries that are the primary focus of this section.

The next result is likely known.
We include a proof for completeness.

\begin{prop}\label{prop:deck-commute}
Let $\rat{f}{X}{Z}$ be a branched cover with a group of continuous fiber-respecting symmetries $G$, and let $\Psi \in \deck (f)$ be any deck transformation.
Then $\Psi$ is $G$-equivariant---that is,
\begin{equation}\label{eq:equivariant-deck}
\Psi (g\cdot x) = g \cdot \Psi (x)
\end{equation}
for all $(g,x)$ in some dense Zariski-open subset of $X \times G$. 
\end{prop}
\begin{proof}
Let $U_G\subset G$, $U_X\subset X$ be nonempty Zariski-open subsets such that both the group action map and $\Psi$ are defined for all pairs in $U_X \times U_G.$
Let $\func{\theta_g}{[0,1]}{U_G}$ be a smooth path in $G$ from $\mathrm{id}_G$ to $g$. 
For any $x \in U_X$, we define a path $U_X$ from $x$ to $g\cdot x$ by the rule
\[ \gamma_{x,g}(t) = \theta_g(t)\cdot x \]
To show the two paths $\gamma_{x,g}$ and $\gamma_{\Psi(x),g}$ have the same projection by $f$, consider the chain of equalities
\[ f(\gamma_{x,g}) = f(\theta_g(t)\cdot x) \overset{(1)}{=} \theta_g(t)\cdot f(x) \overset{(2)}{=} \theta_g(t)\cdot f(\Psi(x)) \overset{(3)}{=} f(\theta_g(t)\cdot \Psi(x)) = f(\gamma_{\Psi(x), g}), \]
where $(1)$ and $(3)$ follow from $G$-equivariance of $f$ and $(2)$ follows from the fact that $f\circ\Psi = f$. 
Now, using~\Cref{cor:deck_after_tracking}, simply note
\[ \Psi(g\cdot x) = \Psi(\gamma_{x,g}(1)) = \gamma_{\Psi(x),g}(1) = g\cdot \Psi(x). \]
\end{proof}

In what follows, we will restrict our attention to special cases where the group $G$ is abelian which leads directly to algorithmic improvements.
However, in view of the ubiquity of equivariant branched covers for more general groups, it would be interesting to study any further applications of this structure to the problem of deck transformation recovery.

\subsection{Scaling symmetries from Smith Normal Forms}\label{subsec:smith}

Consider again a branched cover of problem-solution pairs $f : X \to \CC^m$ as in~\Cref{assumption:sampling-and-equations}, where $X \subset \VV (F) \subset \CC^{n+m}$
is an irreducible affine variety locally defined by the square polynomial system
\begin{equation}
F = \begin{bmatrix}
f_1 (\xx, \pp) \\
\vdots \\
f_n (\xx , \pp)
\end{bmatrix}
=
\begin{bmatrix}
\sum_{i=1}^{k_1} a_{1i}(\xx, \pp)^{\al_{1i}}\\
\vdots \\
\sum_{i=1}^{k_n} a_{n i}(\xx, \pp)^{\al_{ni}} 
\end{bmatrix},\label{eq:F-support}
\end{equation}
with all $a_{i j}$ nonzero.
Each set of exponents $\{ \alpha_{i 1}, \ldots , \alpha_{i k_i} \}$ in~\eqref{eq:F-support} is called the monomial \emph{support} of $f_i,$ and denoted $\supp (f_i).$

For any $\lambda \in \CC^*$ and $\uu \in \ZZ^{n+m}$, and $(\xx , \pp) \in \CC^{n+m},$ we define
\begin{align}
\lambda^\uu = \begin{bmatrix}
\lambda^{u_1} & \cdots & \lambda^{u_{n+m}}
\end{bmatrix}^T &\in \CC^{n+m}, \nonumber \\
\lambda^\uu \str (\xx , \pp) = \begin{bmatrix}
\lambda^{u_1} x_1 & \cdots & \lambda^{u_{n+m}} p_{m}
\end{bmatrix}^T &\in \CC^{n+m}, \nonumber \\
\lambda^\uu \str \pp = \begin{bmatrix}
\lambda^{u_{n+1}} p_1 & \cdots & \lambda^{u_{n+m}} p_{m}
\end{bmatrix}^T &\in \CC^{m}. \label{eq:scale-actions}
\end{align}

We shall be interested in detecting certain scaling symmetries on $X$ of the form $\lambda \cdot (\xx , \pp ) = \lambda^{\uu} \str (\xx, \pp)$.
We consider both continuous symmetries $\CC^* \times X \dashrightarrow X,$ as well as discrete symmetries $\mathbb{Z}_s \times X \dashrightarrow X,$ where $\mathbb{Z}_s$ denotes the integers modulo $s.$ 
In the latter case, this means $\lambda^s=1.$
Such symmetries can be detected using well-known methods of integer linear algebra.
From the system $F$ in~\eqref{eq:F-support}, we form a matrix of shifted exponent vectors
\begin{equation}\label{eq:shifted-exponent-matrix}
\mathbf{A} = \begin{bmatrix}\al_{12}-\al_{11} & \dots & \al_{n k_n}-\al_{n 1} \end{bmatrix} \in \ZZ^{(n+m) \times \sum_{i=1}^n (k_i-1)},
\end{equation}
and compute its Smith Normal Form,
\begin{equation}\label{eq:A-SNF}
\begin{split}
\mathbf{A} = 
\left[ 
\begin{array}{c}
\mathbf{U}_3 \\
\hline 
\mathbf{U}_2 \\
\hline 
\mathbf{U}_1 
\end{array}
\right]^{-1}
\left[ \begin{array}{c|c|c}
\mathbf{I} & \mathbf{0} & \mathbf{0}\\
\mathbf{0}& \mathbf{D} & \mathbf{0}\\
\mathbf{0} & \mathbf{0} & \mathbf{0}
\end{array}\right]
V^{-1},\\
\mathbf{D} = \operatorname{diag} (d_1, \ldots , d_n),
\quad \text{w/ } d_i > 1,  \, \, d_i \, | \,  d_{i+1}.
\end{split}
\end{equation}
The blocking in~\eqref{eq:A-SNF} is chosen so that every row of $\mathbf{U}_1$ lies in the left-nullspace of $\mathbf{A}.$
Similarly, each row of $\mathbf{U_2}$ becomes a left null vector of $\mathbf{A}$ after reduction modulo some $d_i.$
We first consider continuous symmetries arising from $\mathbf{U}_1.$
\begin{defn}\label{def:continuous-scaling-symmetries}
Fix $f, F,$ and $\mathbf{U}_1 \in \ZZ^{r \times (n+m)}$ as above, and consider the rows of $\mathbf{U}_1,$ denoted $\mathbf{u}_1^T, \ldots , \mathbf{u}_r^T \in \ZZ^{1\times (n+m)}$.
We define an associated group of \emph{continuous scaling symmetries} $\Gcon$ to be the image of the  homomorphism 
\begin{align*}\label{eq:homomorphism-free}
(\mathbb{C}^*)^r &\to (\CC^*)^{n+m}\\
(\lambda_1, \ldots , \lambda_r) &\mapsto \lambda_1^{\uu_1^T} \str \cdots \str \lambda_r^{\uu_r^T}.
\end{align*}
\end{defn}
As an abstract group, we have $\Gcon \cong (\CC^*)^r$, and this group acts on $\CC^{n+m}$ and $\CC^m$ via the coordinate-wise product $\str$ as in~\eqref{eq:scale-actions}.
In fact, $\Gcon $ also acts on $X$.
To see this, first observe for any point $(\xx, \pp) \in X \subset \VV (F)$ that orbits $\Gcon \cdot (\xx ,\pp )$ are all contained in $\VV (F)$ by construction.
On the other hand, $(1, \ldots , 1) \cdot (\xx , \pp) \subset X,$ so by connectivity we may conclude $\Gcon \cdot (\xx ,\pp ) \subset X.$
Thus, the branched cover $f$ is equivariant with respect to the action of the connected algebraic group $\Gcon$.
By~\Cref{prop:deck-commute}, it follows that any deck transformation $\Psi \in \deck (f)$ must commute with the action of $\Gcon .$
This implies that the coordinate functions of $\Psi$ are \emph{quasi-homogeneous.}

\begin{defn}\label{def:quasihomogeneous-continuous}
A rational function $\frac{p(\xx)}{q(\xx)} \in \CC(X)$ is said to be quasi-homogeneous with respect to a continuous scaling $\uu$ if there exists $d\in \ZZ$ with
    \[ \frac{p(\lambda^{\uu} \str \xx)}{q(\lambda^{\uu} \str \xx)} = \lambda^d\cdot \frac{p(\xx)}{q(\xx)} \quad \quad \forall \xx \in X, \, \, \forall \lambda \in \CC^{*}. \]
\end{defn}
We may verify that a quasi-homogeneous rational function can be represented as the quotient of two quasi-homogeneous polynomials.
\begin{prop}\label{prop:homogeneous-rational-continuous}
 Consider a nonzero rational function on an irreducible affine variety $X,$ represented as $\frac{p(\xx)}{q(\xx)}$ where $p$ and $q$ are both polynomials that do not vanish on $X.$ 
 If the function is quasi-homogeneous with respect to a free scaling $\uu$, then it can be represented as $\frac{a(\xx)}{b(\xx)}$ on $X$, where both $a(\xx)$ and $b(\xx)$ are both quasi-homogeneous polynomials with respect to $\uu$. 
 Moreover, we may assume all monomials in $a$ (resp.~$b$) occur in $p$ (resp.~$q$),
\[ \supp(a(\xx)) \subseteq \supp(p(\xx)), \quad \supp(b(\xx)) \subseteq \supp(q(\xx)). \]
\end{prop}
\begin{proof}
The univariate Laurent polynomial ring $\CC [\xx ][\lambda^{-1}]$ is equipped with a natural $\mathbb{Z}$-grading with respect to degrees in $\lambda .$
Let
    \[ p(\lambda^\uu \str \xx) = \sum_{i=1}^{h_1} \lambda^{r_i}p_i(\xx), \quad q(\lambda^\uu \str \xx) = \sum_{i=1}^{h_2} \lambda^{s_i}q_i(\xx), \]
    be expressions of $p(\lambda^\uu \str \xx)$ and $q(\lambda^\uu \str \xx)$ in homogeneous components with respect to this grading. Note that the integers $r_i$ (resp. $s_i$) are distinct, and 
    without loss of generality we may assume that none of the $p_i$ or $q_i$ vanish on $X$. From the identities
    \[ \frac{\sum_{i=1}^{h_1} \lambda^{r_i}p_i(\xx)}{\sum_{i=1}^{h_2} \lambda^{s_i}q_i(\xx)} = \frac{p(\lambda^{\uu} \str \xx)}{q(\lambda^{\uu} \str \xx)} = \lambda^d\cdot\frac{p(\xx)}{q(\xx)}, \]
    let us define
    \begin{equation} \label{eq:hom_rat}
    f(\lambda, \xx) := \sum_{i=1}^{h_1}\lambda^{r_i}p_i(\xx)q(\xx) - \sum_{i=1}^{h_2}\lambda^{s_i+d}q_i(\xx)p(\xx) \in \CC [\xx ][\lambda^{-1}].
    \end{equation}
Clearly we have $f(\lambda , \xx)=0$ for all $ \xx \in X$ and $\lambda \in \CC^{*}$.
    Let $t_i = s_i+d$ for all $i \in [h_2]$. Suppose that there exists $k \in [h_1]$ such that for all $j \in [h_2]$ we have $r_k \neq t_j$. Then we can rewrite \eqref{eq:hom_rat} as
    \[ f(\lambda, \xx) = \sum_{\substack{e_i \in \{ r_i \} \cup \{ t_i\}\\ e_i \ne r_k} } \lambda^{e_i}f_i(\xx) + \lambda^{r_k}p_k(\xx)q(\xx) = 0, \quad \forall \xx \in X, \, \forall \lambda \in \CC^{*},
    \]
    for suitable polynomials $f_i(\xx)$ with $e_i$ are all distinct. 
Since a univariate Laurent polynomial that evaluates to zero at infinitely many points is the zero polynomial, we obtain a contradiction
    \[ p_k(\xx)q(\xx) = 0 \;\; \forall \xx \in X \iff p_k(\xx) = 0 \;\; \forall \xx \in X \text{ or } q(\xx) = 0 \;\; \forall \xx \in X. \]
Thus, for every $i \in [h_1]$ there exists $j \in [h_2]$ such that $r_i = t_j$. By symmetry, for every $j \in [h_2]$ there exists $i \in [h_1]$ such that $r_i = t_j$. In particular, $h_1 = h_2$. 
If we consider now the set $
H = \{ (i,j) \in [h_1]^2 \mid r_i = t_j \},$
then we can rewrite~\eqref{eq:hom_rat} as
    \[ f(\lambda, \xx) = \sum_{(i,j) \in H} \lambda^{r_i} \Big(p_i(\xx)q(\xx) - q_j(\xx)p(\xx)\Big) = 0 \]
    Again, since $r_i$ are all distinct, we have
    \[ p_i(\xx)q(\xx) - q_j(\xx)p(\xx) = 0, \quad \forall \xx \in X, \, \forall (i,j) \in H. \]
    Thus, for any $(i,j) \in H,$ the rational function $p(\xx) / q(\xx)$ has an equivalent representation $p_i (\xx) / q_j (\xx)$.
    Taking $(a,b) = (p_i, q_j)$ gives the desired conclusion.
\end{proof}
We now turn our attention to \emph{discrete} scaling symmetries arising from the submatrix $\mathbf{U}_2$ appearing in the Smith Normal Form~\eqref{eq:A-SNF}.
These will form a finite group $\Gdis$, defined analogously to $\Gcon .$ 
However, to do this we must consider some technicalities not present in the continuous case.
Write
\begin{equation}\label{eq:blocking-U2}
\mathbf{U}_2 = 
\begin{bmatrix}
\mathbf{U}_{2,1} \\
\hline
\vdots \\
\hline
\mathbf{U}_{2,k} 
\end{bmatrix},
\end{equation}
where the blocking is chosen so that each $\mathbf{U}_{2,i} \in \ZZ^{r_i \times (n+m)}$ 
corresponds to a distinct elementary divisor $d_i$ in $\mathbf{D}.$

For each of the blocks $\mathbf{U}_{2, i}$ in~\eqref{eq:blocking-U2}, let 
the rows of this matrix after reduction modulo $d_i$
be denoted by $\uu_{i,1}^T, \ldots , \uu_{i,r}^T \in \ZZ_{d_i}^{1\times (n+m)}$.
Write $U_i$ for the group of $d_i$-th roots of unity in $\CC^*$, and consider the homomorphism 
\begin{align}\label{eq:homomorphism-torsion}
\displaystyle\prod_{i=1}^k U_i^{r_i} &\to (\CC^*)^{n+m}\\
\nonumber
\left(
(\lambda_{1,1}, \ldots , \lambda_{1,r_1}),
\ldots , 
(\lambda_{k,1}, \ldots , \lambda_{k,r_k})
\right)
&\mapsto \lambda_1^{\uu_{1,1}^T} \str \cdots \str \lambda_1^{\uu_{1,r_1}^T}
\str \cdots \str
\lambda_k^{\uu_{k,r_k}}.
\end{align}

\begin{defn}\label{def:discrete-scaling-symmetry}
The group of \emph{discrete scaling symmetries} associated to $f$ and $F$ consists of all scalings in the image of~\eqref{eq:homomorphism-torsion} that map $X$ to itself and commute with all deck transformations in $\deck (f).$   
\end{defn}

Abstractly, $\Gdis$ is a finite abelian group, $\Gdis \cong \displaystyle\prod_{i=1}^{k} \ZZ_{d_i}^{r_i}$. 
As in the continuous case, $f$ is $\Gdis$-equivariant.
However, both requirements that elements of $\Gdis$ map $X$ to itself and commute with $\deck (f)$ are now nontrivial.

\begin{example}\label{ex:discrete-symmetry-pathologies}
Consider the system, with $(n,m) = (4,3),$ defined by
\[
F =
\begin{bmatrix}
2 x_1^2 + 1\\
x_{2}+2\,x_{1}x_{3}+p_{1}\\ 
3 x_3^2 - x_4^2 - 4\,p_{1}x_{1}x_{3}-2p_{2}\\
x_1 x_3^3 + 3 x_1 x_3 x_4^2 + p_1 x_3^1 + p_1 x_4^2 - 2 p_2 x_1 x_3 - 2 p_3
\end{bmatrix}.
\]
The basic idea behind constructing the system $F$ in this example is to take a Galois cover with Galois group $S_3,$ and apply a birational change of coordinates so that one of its deck transformations becomes a scaling.
The following Macaulay2~\cite{M2} code was used:
\begin{verbatim}
FF = frac(QQ[p_1..p_3, x_1]/ideal(2*x_1^2+1));
R = FF[x_2..x_4];
I = ideal apply({x_2, x_3, x_4}, v -> v^3 + p_1 * v^2 + p_2 * v + p_3);
J = saturate(I, ideal((x_2-x_3)*(x_2-x_4)*(x_3-x_4)));
phi = map(R, R, {x_2, x_1*(x_3+x_4), x_1*(x_3-x_4)});
K = phi J;
scale = map(R, R, {x_2, x_3, -x_4});
F = apply(K_*, p -> p + scale p)
\end{verbatim}
The variety $\VV (F)\subset \CC^7$ has the irreducible decomposition
\[
\VV (F) = \left(\VV (F) \cap \VV (x_1 + \sqrt{-1/2}) \right) \cup \left( \VV (F) \cap \VV (x_1 - \sqrt{-1/2}) \right).
\]
Let $X$ be the first of these irreducible components.
We see that the scaling on $\CC^7$ that sends $x_1 \to - x_1$ and fixes all other coordinates does not map $X$ to itself.
Moreover, by construction we have that the scaling on $X$ that sends $x_4\to -x_4$ and fixes all other coordinates does not commute with the full deck transformation group $\deck (f) \cong S_3.$
Since neither symmetry belongs to $\Gdis ,$ this shows why both requirements of~\Cref{def:continuous-scaling-symmetries} are necessary.
\end{example}

\Cref{ex:discrete-symmetry-pathologies} underscores the need for a procedure for computing $\Gdis .$
We now summarize a ``probability-one" procedure based on homotopy continuation that accomplishes this task.
For each elementary divisor $d_i$, consider all $\mathbb{Z}_{d_i}$-linear combinations of the modulo-$d_i$ reduced rows of $\mathbf{U}_{2,i}$. For each linear combination $\uu \in \ZZ_{d_i}^{n+m},$ we check if the associated scaling commutes with each deck transformation. This is done with a probability-one homotopy test---generate random intermediate parameter values $\pp_1\in \CC^m$ and track along two linear segment parameter homotopies---first from $\pp_0$ to $\pp_1$, then from $\pp_1$ to $\lambda^\uu \str \pp_0.$\footnote{Intermediate parameters are used because $\pp_0$ and $\lambda^\uu \str \pp_0$ are not in general position with respect to each other. Alternatively, the $\gamma$-trick can be used, cf.~\cite{coeffParam}.}
By~\Cref{cor:deck_after_tracking}, the following holds with probability-one: the discrete scaling $\uu$ commutes with $\Psi \in \deck (f)$ if and only if the endpoint obtained by tracking $\Psi (\xx_0)$ along the homotopies is the same as the endpoint obtained by
first tracking the start point $\xx_0$ and then applying $\Psi $.
This probability-one test succeeds for all $\Psi \in \deck (f)$ if and only if $\uu$ determines an element of $\Gdis .$

\subsection{Quasi-homogeneous interpolation}\label{subsec:quasi-interp}

Retaining the notation established earlier in the section, we now consider the \emph{multidegree map}, ie.~the group homomorphism $\ZZ^{n+m} \to \ZZ^r \times \prod_{i=1}^k \ZZ_{d_i}^{r_i'}$ specified by matrices
\begin{equation}\label{eq:compatible-Us} (\mathbf{U}_1, \mathbf{U}_{2,1}', \ldots, \mathbf{U}_{2,k}') \in \ZZ^{r\times(n+m)} \times \prod_{i=1}^k \ZZ_{d_i}^{r_i'\times(n+m)}, \end{equation}
where $r_i' \le r_i$ and the $\ZZ_{d_i}$-rowspan of each $\mathbf{U}_{2,i}'$ is contained in the $\ZZ_{d_i}$-rowspan of $\mathbf{U}_{2,i}.$
We assume that the multidegree map is \emph{compatible} with $f$ in the sense that the deck transformations commute with the group of scaling symmetries in $\Gcon \times \Gdis$, obtained by applying $\operatorname{Hom} (\bullet , \CC^*)$ to the image of~\eqref{eq:compatible-Us}.
The procedure for computing $\Gdis $ described above allows us to easily determine a set of maximally compatible $\mathbf{U}_{2,1}', \ldots, \mathbf{U}_{2,k}'$ from $\mathbf{U}_{2,1}, \ldots , \mathbf{U}_{2,k}.$

\begin{algorithm}
\small 
\DontPrintSemicolon
  \KwInput{$F = (f_1, \ldots , f_n)$, $(\xx^*, \pp^*)$, $D^*$ as in~\Cref{alg:deck-interpolation}, and the $f$-compatible multidegree map specified by $\mathbf{U}=(\mathbf{U}_1, \mathbf{U}_{2,1}', \ldots, \mathbf{U}_{2,k}')$ as in~\eqref{eq:compatible-Us}.}
  \KwOutput{As in~\Cref{alg:deck-interpolation}}
  $(x^{(1)}, \dots, x^{(d)}), \mathrm{Mon}(f,\pp^*) \gets$ \texttt{run\char`_monodromy}($F, \xx^*, \pp^*$)\\
  $\{\sigma_1,\dots,\sigma_q\} \gets \mathrm{Cent}_{S_d}(\mathrm{Mon}(f, \pp^*))$ \tcp*[f]{WLOG $q>1$, $\sigma_1 = \textrm{id}$}
  \For{$i \gets 2;\ i \leq q;\ i \gets i + 1$}{
    $\Psi_i \gets \matrix{\textbf{missing} & \dots & \textbf{missing}}^\top$
  }
  \For{$D \gets 1;\ D \leq D^*;\ D \gets D + 1$}{
    $M \gets$ monomials($\xx$, $\pp$, $D$) \tcp*{up to total degree $D$}
    $C \gets$ classes($M, \mathbf{U}$) \tcp*{dictionary of monomial classes}
    $t \gets \max\, \{ \# m \;|\; (d, m) \in C \}$ \tcp*{size of the largest class}
    Track the orbit $\deck(f) \cdot x^{(1)}$ to $2t$ random instances of $F$
    \For(\tcp*[f]{iterate through all classes}){($deg_n, mon_n$) $\in$ C}{
        $t_n \gets \# mon_n$ \tcp*{size of the class for the numerator}
        \For{$i \gets 1;\ i \leq n;\ i \gets i + 1$}{
          $deg_d \gets$ \texttt{denominator\char`_degree}($\mathbf{U}, deg_n, i$) 
        \If{$deg_d$ $\in$ keys(C)}{
        $mon_d \gets$ \texttt{get\char`_value}($C, deg_d$)\\
        $t_d \gets \# mon_d$}
          \For{$j \gets 2;\ j \leq q;\ j \gets j + 1$} {
          \If{$\Psi_{j_i} \mathrm{\;is \; \mathbf{missing}}$}{
            $\mathbf{A} \gets$ $(t_n+t_d) \times (t_n+t_d)$ Vandermonde matrix\\
            \indent $\quad \xx_k ' = \sigma_j \cdot \xx_k$ for $k=1,\ldots , t_n+t_d$ \\
            $\mathbf{N} \gets$ nullspace($\mathbf{A}$)\\
            $\wt{\mathbf{N}} \gets$ rref($\mathbf{N}^\top$)\\
            $\matrix{\aa^\top & \bb^\top}^\top \gets$ \texttt{get\char`_representative}$(\wt{\mathbf{N}})$
            \If{$\matrix{\aa^\top & \bb^\top}^\top \mathrm{is\; not \; \mathbf{nothing}}$}{
              $\Psi_{j_i} \gets \frac{\sum_{k=1}^{t_n} a_k\cdot(\xx,\pp)^{\al_k}}{\sum_{k=1}^{t_d} b_k\cdot(\xx,\pp)^{\be_k}}$
            }
          }
          }
        }
        }
                \If{\text{all} $\Psi_i$ \text{are interpolated}}{
        \Return $\{\Psi_1,\dots,\Psi_q\}$
    }
  }
  \Return $\{\Psi_1,\dots,\Psi_q\}$
\caption{Quasi-homogeneous interpolation of $\deck (f)$}\label{alg:graded-deck-interpolation}
\end{algorithm}

Our approach to interpolation of quasi-homogeneous deck transformations is summarized by the pseudocode of~\Cref{alg:graded-deck-interpolation}.
The overall structure is much the same as~\Cref{alg:deck-interpolation}.
The main difference, responsible for the improved performance, is that a potentially much smaller basis of monomials for interpolation is chosen, so as to incorporate the quasi-homogeneity of $\deck (f)$ encoded by the symmetry group $\Gcon \times \Gdis .$
An even more refined strategy, not pursued here, would be to consider quasi-homogeneity at the level of individual deck transformations, or even their coordinate functions.
Nevertheless, experiments in~\Cref{sec:ex} show that working with just the symmetries in $\Gcon \times \Gdis $ already produces considerable savings.

On line 8 of~\Cref{alg:graded-deck-interpolation}, the set of all monomials of degree $\le D^*$ is divided into monomial \emph{classes}.
The sizes of these classes govern complexity of each interpolation step. 
Each class consists of key-value pairs, where the values are monomials whose corresponding key records the multidegree of the numerator of some candidate rational function representation.
The multidegree of the denominator in such a representation is uniquely determined by the class.
Pseudocode for computing the multidegree of the denominator from a given class is provided in~\Cref{alg:numer_deg2denom_deg} (according to the $\Gcon \times \Gdis$-equivariance of $\deck(f)$~\eqref{eq:equivariant-deck}).

\begin{algorithm}[H]
\DontPrintSemicolon
  \KwInput{Matrices $\mathbf{U}_1, \mathbf{U}_{2,1}', \ldots, \mathbf{U}_{2,k}' \in \ZZ^{r\times(n+m)} \times \prod_{i=1}^k \ZZ_{d_i}^{r_i'\times(n+m)}$ specifying a multidegree map compatible with $f$, a candidate numerator multidegree $(\mathbf{n}_0, \mathbf{n}_1, \dots, \mathbf{n}_k) \in \ZZ^r \times \prod_{i=1}^k \ZZ_{d_i}^{r_i'}$, and $j$ indexing the rational function $\psi_j$ in $\Psi \in \deck (f)$}
  \KwOutput{The corresponding denominator multidegree, $(\mathbf{d}_0, \mathbf{d}_1, \dots, \mathbf{d}_k) \in \ZZ^r \times \prod_{i=1}^k \ZZ_{d_i}^{r_i'}$}
  \Return $(\mathbf{n}_0- (\mathbf{U}_1)_{[:, j]}, \, \mathbf{n}_1-(\mathbf{U}_{2,1}')_{[:, j]}, \, \dots, \,  \mathbf{n}_k-(\mathbf{U}_{2,k}')_{[:, j]})$

\caption{Subroutine called on line 14 of~\Cref{alg:graded-deck-interpolation}.}\label{alg:numer_deg2denom_deg}
\end{algorithm}

\section{Examples and Experiments}\label{sec:ex}


Our implementation of~\Cref{alg:deck-interpolation,alg:graded-deck-interpolation} is written in the Julia programming language. 
It depends on the following packages: 
\begin{enumerate}
    \item \texttt{HomotopyContinuation.jl}\cite{HCJL},
    \item \texttt{AbstractAlgebra.jl}, part of the Nemo system~\cite{nemo}, and
    \item \texttt{GAP.jl}, part of the OSCAR system~\cite{OSCAR}.
\end{enumerate}
Basic instructions for using the code (which is currently under active development) can be found at the link:
{\textcolor{magenta}{\url{https://multivariatepolynomialsystems.github.io/DecomposingPolynomialSystems.jl/dev/}}}.

All timings reported were obtained with a 2022 Mac M1 with 8GB RAM.

\subsection{Five-point relative pose}\label{subsec:five-point-problem}

One of the most well-known minimal problems in computer vision is the classical five-point problem as shown in~\Cref{fig:5pp}. While many solvers exist for this problem~\cite{nister}, and the symmetry is well-known, this section aims to show how the methods in this paper can recover this symmetry without any \emph{a priori} knowledge.

\begin{figure}
\centering
\includegraphics[scale=0.8]{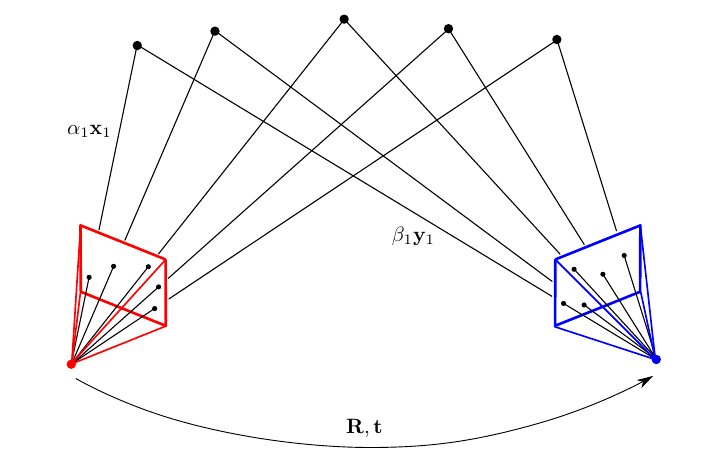}
\caption{Geometry of  five-point relative pose.}
\label{fig:5pp}
\end{figure}

We consider two slightly different formulations of the five-point problem.
\Cref{subsubsec:5pp-inhomog} presents an ``inhomogeneous" formulation that appeared previously in the conference paper~\cite{us_ISSAC23}.
For this formulation, we were able to recover only the coordinate functions of the deck transformation which are parameter-independent.
\Cref{subsubsec:5pp-quasihomog} shows how incorporating scaling symmetries in~\Cref{alg:graded-deck-interpolation} leads to a much more tractable problem, in which the full deck transformation can be recovered.
A summary of these experiments is given in~\Cref{tab:5pp} below---we refer to the corresponding subsections for details.

\begin{table}[H]
    \footnotesize
    \centering
    \begin{tabular}{l  c c c} \toprule
         & Inhom & Quasi-hom & Quasi-hom \\
         & $\&$ & $\&$ & $\&$ \\
         & Param-indep & Param-indep & Param-dep \\
        \midrule
        Formulation & \eqref{eq:5pp-inhom} & \eqref{eq:5pp-quasi-hom} & \eqref{eq:5pp-quasi-hom} \\
        Number of solutions & 20 & 20 & 20 \\
        Degree bound $D^*$ & 3 & 3 & 3 \\
        Monodromy time & 5sec & 5sec & 5sec \\
        Tracking time & 1min & $<$1sec & $<$1sec \\
        Interpolation time & 20min & $<$1sec & 5sec \\
        Largest Vandermonde matrix & $4600 \times 4600$ & $32 \times 32$ & $72 \times 72$ \\
        Number of tracked paths & $2 \times 4600$ & $2 \times 32$ & $2 \times 72$\\ 
        Depths $\alpha_1, \ldots , \beta_5$ recovered? & no & no & yes \\
        \bottomrule
    \end{tabular}
    \caption{Timings/details for recovering deck transformation of the five-point problem.}
    \label{tab:5pp}
\end{table}

\subsubsection{Inhomogeneous formulation}\label{subsubsec:5pp-inhomog}

We first consider the following setup.
There are $5$ correspondences between $2$D image points $\xx_1 \leftrightarrow \yy_1, \ldots , \xx_5 \leftrightarrow \yy_5$. These $2$D data points are $2\times 1$ vectors, and are assumed to be images of $5$ world points under two calibrated cameras, where the two camera frames differ by a rotation $\R$ and a translation $\tran.$  The task for this problem is to solve for the relative orientation $\cam{\R}{\tran} \in \SERR (3) $ between the two cameras and each of the five points in $3$D space, as measured by their depths with respect to the first and second camera frames.

Writing $\alpha_1 , \ldots , \alpha_5$ for the depths with respect to the first camera and $\beta_1,\ldots , \beta_5$ for the depths with respect to the second camera, solutions to the five-point problem must satisfy a system of polynomial equations and inequations:
\begin{equation} \label{eq:5pp-inhom}
    \begin{split}
        \R^\top\R = \I, \quad \det \R = 1,
        \\
        \beta_i\matrix{\yy_i \\ 1} = \R\alpha_i\matrix{\xx_i \\ 1} + \mathbf{t}, \quad \forall \,  i = 1,\dots,5,
        \\
        \mathbf{t} \neq \mathbf{0}.
    \end{split}
\end{equation}

The unknown depths and translation $\mathbf{t}$ are defined in projective space, meaning $\mathbf{t}, \alpha_1, \ldots , \alpha_5, \beta_1, \ldots , \beta_5$ can only be recovered up to a common scale factor.  One option to remove this ambiguity is to treat these unknowns as homogeneous coordinates on a $12$-dimensional projective space, then for generic data $\xx_1,\ldots , \xx_5, \yy_1, \ldots , \yy_5,$ there are at most finitely many solutions in $(\R, \mathbf{t}, \alpha_1, \ldots, \alpha_5, \beta_1,\ldots , \beta_5) \in \SOCC (3) \times \PP_\CC^{12}$ to the system~\eqref{eq:5pp-inhom}.  This finiteness is what creates the minimal problem structure.  In practice, these solutions may be computed by working in a fixed affine patch of $\PP = \PP_\CC^{12}$ such that the inequality $\t \neq \mathbf{0}$ is satisfied (e.g., $\mathbf{a}^\top\t = 1$ for a random $\mathbf{a} \in \CC^3$).

There are exactly $20$ solutions over the complex numbers for generic data in $Z = \left( \CC^2 \right)^5 \times \left( \CC^2 \right)^5.$ 
The solutions to~\eqref{eq:5pp-inhom} are naturally identified with the fibers of a branched cover $f:X \to Z$, where 
\begin{equation*}
    \begin{split}
        X = \{ \left(\R, (\mathbf{t}, \alpha_1, \ldots , \beta_5), (\xx_1,\ldots ,  \ldots , \yy_5) \right)
        \in \SOCC (3) \times \PP_\CC^{12} \times Z \mid 
\eqref{eq:5pp-inhom} \text{ holds }\}.
    \end{split}
\end{equation*}
With our chosen formulation, the branched cover $f$ has a single deck transformation $\Psi$ known as the \emph{twisted pair}, whose coordinate functions are

\begin{equation}
\label{eq:twisted-pair}
\begin{split}
\Psi_\R &= \left( \frac{2}{\normsq{\tran}} \tran \tran^\top \, - \I\right)\, \R    ,
\\
\Psi_\t &= \t ,
\\
\Psi_{\alpha_i} &= \displaystyle\frac{-\alpha_i \norm{\tran}^2 }{2 \, \langle \t, \beta_i\yy_i \rangle - \normsq{\tran}} ,
\\
\Psi_{\beta_i} &= 
\displaystyle\frac{\beta_i \norm{\tran}^2 }{2 \, \langle \t, \beta_i\yy_i \rangle - \norm{\tran}^2} ,
\\
\Psi_{\xx_i} &= \xx_i ,
\\
\Psi_{\yy_i} &= \yy_i .
\end{split}    
\end{equation}

\begin{figure}
\resizebox{\textwidth}{!}{\includegraphics{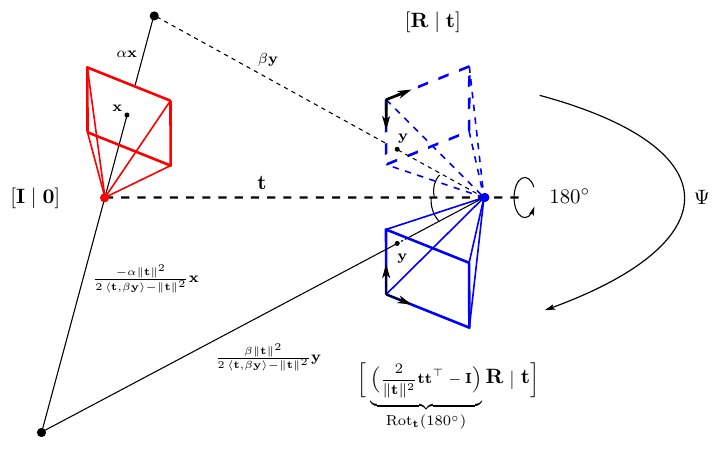}}
\caption{Twisted pair symmetry of the five-point problem.}
\label{fig:twisted_pair}
\end{figure}

We see that $\Psi $ consists of coordinate functions of total degree at most $3.$
The effect of this deck transformation on solutions to the five-point problem is illustrated in~\Cref{fig:twisted_pair}.
The coordinate functions $\Psi_{\R}, \Psi_{\t}$ are parameter-independent, whereas the $\Psi_{\alpha_i}, \Psi_{\beta_i}$ are parameter-\emph{dependent.}

We ran~\Cref{alg:deck-interpolation} on the formulation \eqref{eq:5pp-inhom} with the upper bound for the total degree $D^* = 3$. 
However, when running~\Cref{alg:deck-interpolation}, we considered only the parameter-independent setting, for which $t = {22 + 3 \choose 3} = 2300$.
In the parameter-dependent setup, we would have $2t = 2{22 + 20 + 3 \choose 3} = 28380$ coefficients to interpolate.
This exceeded the capacity of our machine.

The computation described above succeeded in recovering $\Psi_{\R}$ and $\Psi_{\t}$ in~\eqref{eq:twisted-pair}.
For the coordinate functions $\Psi_{\alpha_i}, \Psi_{\beta_i},$  no reasonable representative was found---all rows of $\wt{\mathbf{N}}$ were such that $\aa \approx \mathbf{0}$ or $\bb \approx \mathbf{0}.$ 
These coordinate functions remain ``missing'' in~\Cref{alg:deck-interpolation}.
This is expected---the expressions for these functions in~\eqref{eq:twisted-pair} are parameter-dependent.

Further details for this formulation are given in~\Cref{tab:5pp}.


\subsubsection{Quasi-homogeneous formulation}\label{subsubsec:5pp-quasihomog}

In the quasi-homogeneous formulation we consider the image points to be the points in $\PP^2$, i.e. we introduce 10 more parameters and consider the system of polynomial equations and inequations:
\begin{equation} \label{eq:5pp-quasi-hom}
    \begin{split}
        \R^\top\R = \I, \quad \det \R = 1,
        \\
        \beta_i\yy_i = \R\alpha_i\xx_i + \mathbf{t}, \quad \forall \,  i = 1,\dots,5,
        \\
        \mathbf{t} \neq \mathbf{0}.
    \end{split}
\end{equation}
where $\xx_1, \ldots, \yy_5$ are now $3 \times 1$ parameter vectors.
The continuous scaling symmetries of the formulation \eqref{eq:5pp-quasi-hom} are given by

\begin{equation} \label{eq:5pp-free}
    \begin{split}
        (\alpha_i, \xx_i) &\mapsto (\lambda\cdot\alpha_i, \lambda^{-1}\cdot\xx_i), \\
        (\beta_i, \yy_i) &\mapsto (\lambda\cdot\beta_i, \lambda^{-1}\cdot\yy_i), \\
        (\t, \alpha_1, \ldots , \beta_5) &\mapsto \lambda\cdot(\t, \alpha_1, \ldots , \beta_5),
    \end{split}
\end{equation}
where $i=1, \ldots , 5,$ and hence $\Gcon \cong (\CC^*)^{11}.$
The free scalings are represented by $\mathbf{U}_1 \in \ZZ^{11\times52}$. 
Also, $\Gdis \cong \ZZ_2^4,$ with generators

\begin{equation} \label{eq:5pp-discrete}
    \begin{split}
        (\R, \t, \yy_1, \ldots, \yy_5) &\mapsto (\R_1\R, \R_1\t, \R_1\yy_1, \ldots, \R_1\yy_5), \\
        (\R, \xx_1, \ldots, \xx_5) &\mapsto (\R\R_2, \R_2^{-1}\xx_1, \ldots, \R_2^{-1}\xx_5), \\
        \R_1, \R_2 \in \bigg\{ \I, &\begin{bsmallmatrix*}[r] -1&0&0\\0&-1&0\\0&0&1 \end{bsmallmatrix*}, \begin{bsmallmatrix*}[r] -1&0&0\\0&1&0\\0&0&-1 \end{bsmallmatrix*}, \begin{bsmallmatrix*}[r] 1&0&0\\0&-1&0\\0&0&-1 \end{bsmallmatrix*} \bigg\}.
    \end{split}
\end{equation} 
In this example, all discrete scalings detected by the Smith Normal Form commute with $\Psi$.
Hence $r_1 = r_1'$ in the notation of~\Cref{subsec:quasi-interp}.
This commutativity can be detected with the probability-one homotopy test described in~\Cref{subsec:smith}.
This can also be understood \emph{a priori} because these symmetries are instances of a continuous, non-scaling symmetry expressed as in~\eqref{eq:5pp-discrete}, where $\R_1, \R_2 \in \SOCC (3)$ may be arbitrary rotations. 

The last column of~\Cref{tab:5pp} shows that the quasi-homogeneous approach allows us to run parameter-dependent interpolation of the twisted pair~\eqref{eq:twisted-pair}. Using line 8 of~\Cref{alg:graded-deck-interpolation} we partition the monomials up to total degree $3$ in both unknowns and parameters into classes w.r.t. the multidegree map given by $(\mathbf{U}_1, \mathbf{U}_{2,1}')$. We obtain $14339$ classes of monomials, the largest of which has size $36$. This partitioning makes parameter-dependent interpolation feasible.

We close our discussion of the five-point problem with a remark that the depths $\alpha_i, \beta_i$ and world points $\XX_i$ can easily be recovered from known rotations and translations $(\R , \t)$ (assuming generic data) using linear algebra.
In the homogeneous formulation, this is based on the relation
\begin{equation}\label{eq:fpp-focal}
\begin{bmatrix}
\mathbf{I} & \mathbf{0} & \xx_i & \mathbf{0} \\
\R & \t & \mathbf{0} & \yy_i
\end{bmatrix}
\begin{bmatrix}
\XX_i \\
-\alpha_i \\
-\beta_i
\end{bmatrix}
= \mathbf{0}.
\end{equation}

Thus, the five-point problem illustrates that interpolating deck transformations may be easier after eliminating certain variables.
On the other hand, our success with interpolating the twisted pair on depths and world points showcases the utility of techniques that exploit quasi-homogeneity, like~\Cref{alg:graded-deck-interpolation}, and selecting a formulation amenable to these techniques.


\subsection{Perspective 3-Point}\label{subsec:p3p}

Another famous algebraic problem associated with computer vision is the so-called P3P problem.
In fact, early studies of this problem date back centuries to Lagrange and Grunert---see~\cite{sturm2011historical} for a more complete history.

Similar to our discussion of the five-point problem, we may consider an inhomogeneous formulation of the problem as in~\Cref{subsubsec:5pp-inhomog},
\begin{equation} \label{eq:p3p-inhom}
    \begin{split}
        \R^\top\R = \I, \quad \det \R = 1,
        \\
        \alpha_i\matrix{\xx_i\\1} = \matrix{\R & \t}\matrix{\XX_i\\1}, \quad \forall \,  i = 1,\dots,3,
        \\
        \mathbf{n}^\top\XX_i = 1, \quad \forall \,  i = 1,\dots,3,
    \end{split}
\end{equation}
and a quasi-homogeneous formulation,
\begin{equation} \label{eq:p3p-quasi-hom}
    \begin{split}
        \R^\top\R = \I, \quad \det \R = 1,
        \\
        \alpha_i\xx_i = \matrix{\R & \t}\XX_i, \quad \forall \,  i = 1,\dots,3,
        \\
        \matrix{\mathbf{n}^\top & 1}\XX_i = 0, \quad \forall \,  i = 1,\dots,3.
    \end{split}
\end{equation}
In the above, the parameters consist three image points $\xx_i$ (in $\CC^2 $ or $\PP^2,$ respectively) and three world points $\XX_i$ (in $\CC^3$ or $\PP^3.$)

Unlike the five-point problem, in P3P the world points are known---this makes the former a \emph{relative pose} problem, and the latter an \emph{absolute pose} problem.
In addition to the unknwon rotations, translations, depths, and world points, there is an additional unknown vector $\mathbf{n}\in \CC^{3\times 1}$ which defines the normal to the plane spanned by $\XX_1, \ldots , \XX_3.$
We include the normal in our formulation because it reduces the total degree of the problem's single deck transformation, which is given by
\begin{equation}
\label{eq:p3p-deck-transfo}
\begin{split}
\Psi_\R &= \R\, \left( \frac{2}{\norm{\mathbf{n}}^2} \mathbf{n}\mathbf{n}^\top \, - \I\right)    ,
\\
\Psi_\t &= \frac{2}{\norm{\mathbf{n}}^2}\R\mathbf{n} - \t ,
\\
\Psi_{\alpha_i} &= -\alpha_i ,
\\
\Psi_{\mathbf{n}} &= \mathbf{n} ,
\\
\Psi_{\xx_i} &= \xx_i ,
\\
\Psi_{\XX_i} &= \XX_i .
\end{split}    
\end{equation}

The formulation \eqref{eq:p3p-quasi-hom} allows us to decompose the monomials into smaller classes, since its group of scalings is larger that of \eqref{eq:p3p-inhom}. Its continuous part is isomorphic to $(\CC^*)^7$ and is given by

\begin{equation} \label{eq:p3p-free}
    \begin{split}
        (\alpha_i, \xx_i) &\mapsto (\lambda\cdot\alpha_i, \lambda^{-1}\cdot\xx_i), \\
        (\xx_i, \XX_i) &\mapsto (\lambda\cdot\xx_i, \lambda\cdot\XX_i), \\
        (\t, \mathbf{n}, X_{1,4}, X_{2,4}, X_{3,4}) &\mapsto (\lambda\cdot\t, \lambda^{-1}\cdot\mathbf{n}, \lambda^{-1}\cdot X_{1,4}, \lambda^{-1}\cdot X_{2,4}, \lambda^{-1}\cdot X_{3,4}),
    \end{split}
\end{equation}
The discrete part has generators analogous to \eqref{eq:5pp-discrete}:

\begin{equation} \label{eq:p3p-discrete}
    \begin{split}
        (\R,\t,\xx_1,\ldots,\xx_3) &\mapsto (\R_1\R, \R_1\t, \R_1\xx_1, \ldots, \R_1\xx_3), \\
        (\R, \XX_{1,1:3},\ldots,\XX_{3,1:3}) &\mapsto (\R\R_2, \R_2^{-1}\XX_{1,1:3}, \ldots, \R_2^{-1}\XX_{3,1:3}), \\
        \R_1, \R_2 \in \bigg\{ \I, &\begin{bsmallmatrix*}[r] -1&0&0\\0&-1&0\\0&0&1 \end{bsmallmatrix*}, \begin{bsmallmatrix*}[r] -1&0&0\\0&1&0\\0&0&-1 \end{bsmallmatrix*}, \begin{bsmallmatrix*}[r] 1&0&0\\0&-1&0\\0&0&-1 \end{bsmallmatrix*} \bigg\}.
    \end{split}
\end{equation}
Thus, $\Gcon \cong (\CC^*)^7$ and $\Gdis \cong \ZZ_2^4.$

\begin{table}[H]
    \footnotesize
    \centering
    \begin{tabular}{l c c c} \toprule
         & Inhom $\&$ & Quasi-hom $\&$ & Quasi-hom $\&$ \\
         & Param-indep & Param-indep & Param-dep \\
        \midrule
        Formulation & \eqref{eq:p3p-inhom} & \eqref{eq:p3p-quasi-hom} & \eqref{eq:p3p-quasi-hom} \\
        Number of solutions & 8 & 8 & 8 \\
        Degree bound $D^*$ & 3 & 3 & 3 \\
        Monodromy time & 1sec & 1sec & 1sec \\
        Tracking time & 7sec & $<$1sec & $<$1sec \\
        Interpolation time & 150sec & 1sec & 6sec \\
        Largest Vandermonde matrix & $2660 \times 2660$ & $50 \times 50$ & $50 \times 50$ \\
        Number of tracked paths & $2 \times 2660$ & $2 \times 50$ & $2 \times 50$ \\
        $\Psi$ fully recovered? & yes & yes & yes \\
        \bottomrule
    \end{tabular}
    \caption{Timings/details for recovering deck transformation of the P3P problem.}
    \label{tab:p3p}
\end{table}

\Cref{tab:p3p} summarizes results of running~\Cref{alg:deck-interpolation} and~\Cref{alg:graded-deck-interpolation} on the P3P formulations~\eqref{eq:p3p-inhom}and~\eqref{eq:p3p-quasi-hom}.
Both algorithms were able to successfully interpolate~\eqref{eq:p3p-deck-transfo} using the \emph{a priori} knowledge that this deck transformation is parameter-independent.
As expected, running~\Cref{alg:graded-deck-interpolation} on the quasi-homogeneous formulation is more efficient.
Additionally,~\Cref{alg:graded-deck-interpolation}, unlike~\Cref{alg:deck-interpolation}, succeeds when run with the parameter-dependent setting.
Intriguingly, the largest monomial class is the same in both parameter dependent and independent cases.

Once again, we find the importance of a carefully-chosen formulation can be key to recovering deck transformations.
One initial experiment in which $\mathbf{n}$ was eliminated from the quasi-homogeneous formulation~\Cref{eq:p3p-quasi-hom}, in which case the deck transformation has coordinate functions of total degree $D^*=5.$
Although the sizes of Vandermonde matrices returned by~\Cref{alg:graded-deck-interpolation} appeared to be reasonable, we were not able to robustly interpolate these degree-$5$ multivariate rational functions.  
We speculate this is due to the well-known fact that large Vandermonde matrices are ill-conditioned~\cite{Pan_Vandermonde,Higham}.

\subsection{Radial camera relative pose}\label{subsec:radial}


Another problem in computer vision that has been recently tackled~\cite{Hruby_2023_CVPR} is the problem of 3D reconstruction or relative pose from 4 images made by a radial camera. 
A radial camera may be understood as a linear map $P: \PP^3 \dashrightarrow \PP^1,$ thus given by a $2\times4$ matrix. 

The radial camera $P$ associates a world point in $\PP^3$ with the \emph{radial line} passing through the center of distortion in an image and the projection of the world point under the usual pinhole model (as for the five-point relative pose problem). The center of distortion may be assumed to be $[0:0:1]\in \PP^2$, so that the equation of the radial line is parametrized by the projected image point $[u : v : 1]\in \PP^2$ as a direction vector, thus giving a point $\mathbf{l} = [u : v] \in \PP^1.$
With these assumptions, a pinhole camera $P_{\text{pin}} : \PP^3 \dashrightarrow \PP^2$ can be associated with a radial camera $P$ as follows:
\begin{equation}\label{eq:pinhole-2-radial}
P = \begin{bmatrix}
1 & 0 & 0\\
0 & 1 & 0
\end{bmatrix} \cdot P_{\text{pin}}.
\end{equation}
Assuming that $P_{\textrm{pin}}$ is calibrated results in the radial camera matrix
\[ \mathbf{P} = \matrix{\mathrm{Cay}(x,y,z) & \mathbf{t}} \in \CC^{2\times4}, \]
where
\begin{equation}
\begin{split}
 \mathrm{Cay} : \CC^3 &\dashrightarrow \CC^{2\times 3} \\
(x,y,z) &\mapsto 
\left[\begin{smallmatrix}
\frac{1 + x^2 - (y^2 + z^2)}{1+x^2+y^2+z^2} & \frac{2 (xy-z)}{1+x^2+y^2+z^2} & \frac{2 (xz + y)}{1+x^2+y^2+z^2}\\
\frac{2(xy+z)}{1+x^2+y^2+z^2} & \frac{1 + y^2 - (x^2 + z^2)}{{1+x^2+y^2+z^2}} & \frac{2 (yz - x)}{1+x^2+y^2+z^2}
\end{smallmatrix}\right].
\end{split}
\end{equation}
is the Cayley parametrization of the first 2 rows of a $3\times3$ camera rotation matrix and $\t \in \CC^2$ are the first 2 elements of the camera translation vector. As explained in~\cite{Hruby_2023_CVPR}, it is enough to have $4$ cameras and $13$ world points to achieve finite number of solutions. The problem is then formulated by
\[ \alpha_{ij}\mathbf{l}_{ij} = \mathbf{P}_j\matrix{\XX_i \\ 1}, \quad i = 1,\ldots,13, \;\; j = 1,\ldots 4. \]

We may choose a world coordinate system by fixing (according to~\cite[Section 3.2]{Hruby_2023_CVPR})
\[ \mathbf{P}_1 = \matrix{1 & 0 & 0 & 0 \\ 0 & 1 & 0 & 0}, \mathbf{P}_2 = \matrix{\mathrm{Cay}(x_2,y_2,z_2) & \mathbf{e}_2}. \]
(Here $\mathbf{e}_2 = \matrix{0 & 1}^\top$.) We may eliminate the first two unknowns for every world point $\XX_i$ using the relation
\[ \alpha_{i1}\mathbf{l}_{i1} = \mathbf{P}_1\matrix{\XX_i \\ 1} \iff \alpha_{i1}\mathbf{l}_{i1} = \matrix{X_{i1} \\ X_{i2}} \]
we obtain A formulation with $3584$ complex solutions:
\begin{equation} \label{eq:4v-radial}
        \alpha_{ij}\mathbf{l}_{ij} = \mathbf{P}_j\matrix{\alpha_{i1}\mathbf{l}_{i1} \\ X_i \\ 1}, \quad \forall i = 1,\ldots,13, 
        \, \, 
        \forall j = 2,\ldots,4.
\end{equation}
Here $X_i$ denotes the third coordinate of $\XX_i$. The enormous amount of solutions indicates that the problem has to be checked for decomposability (symmetry existence) in order to design for it more efficient solvers~\cite{Hruby_2023_CVPR}.

The continuous scaling symmetries of the formulation \eqref{eq:4v-radial} are given by:
\[ (\alpha_{ij}, \mathbf{l}_{ij}) \mapsto (\lambda\cdot\alpha_{ij}, \lambda^{-1}\cdot\mathbf{l}_{ij}) \]
where $i=1, \ldots , 13,$ $j = 1,\ldots,4$ and hence $\Gcon \cong (\CC^*)^{52}$. The group of discrete scalings is isomorphic to $\ZZ_2^2$. As explained in~\cite{Hruby_2023_CVPR}, exploiting these symmetries and further structure of the branched cover allows for a more practical solution than naively tracking 3584 homotopy paths.
The group of deck transformations of this problem is isomorphic to 
\[ \ZZ_2 \times \ZZ_2 \times \ZZ_2 \times \ZZ_2 \]
By running \Cref{alg:graded-deck-interpolation} we were able to recover all 16 elements of $\deck (f)$.
\begin{table}[H]
    \footnotesize
    \centering
    \begin{tabular}{l c c} \toprule
         & Quasi-hom & Quasi-hom \\
         & $\&$ & $\&$ \\
         & Param-indep & Param-dep \\
        \midrule
        Number of solutions & 3584 & 3584 \\
        Degree bound $D^*$ & 2 & 2 \\
        Monodromy time & 1h & 1h \\
        Tracking time & 45min & 60min \\
        Interpolation time & 15sec & 20sec \\
        Largest Vandermonde matrix & $326 \times 326$ & $430 \times 430$  \\
        Number of tracked paths & $16 \times 326$ & $16 \times 430$ \\
        $\Psi_1,\ldots,\Psi_4$ fully recovered? & yes & yes \\
        \bottomrule
    \end{tabular}
    \caption{Timings/details for recovering deck transformations of the radial relative pose problem.}
    \label{tab:4v-radial}
\end{table}
The details of this experiment are summarized in~\Cref{tab:4v-radial}.
As we can see from~\Cref{tab:4v-radial}, running parameter-dependent version of~\Cref{alg:graded-deck-interpolation} (even though every deck transformation is parameter-independent) results into larger Vandermonde matrices, which in turn forces the algorithm to track more paths. We also notice that running monodromy and tracking paths is the bottleneck for this problem, which may be attributed to a large number of variables and solutions.

In the end, we recover the following four generators of $\deck(f)$,
\begin{align*}
\Psi_1( &x_2,x_3,x_4,y_2,y_3,y_4,z_2,z_3,z_4,\t_3,\t_4,\alpha_{i1},\alpha_{i2},\alpha_{i3},\alpha_{i4},X_i) \\
&= \bigg(-x_2,-x_3,-x_4,-y_2,-y_3,-y_4,z_2,z_3,z_4,\alpha_{i1},\alpha_{i2},\alpha_{i3},\alpha_{i4},-X_i\bigg), \\ \Psi_2(&x_2,x_3,x_4,y_2,y_3,y_4,z_2,z_3,z_4,\t_3,\t_4,\alpha_{i1}, \alpha_{i2}, \alpha_{i3}, \alpha_{i4}, X_i)\\ 
&= \left(\frac{y_2}{z_2},x_3,x_4,-\frac{x_2}{z_2},y_3,y_4,-\frac{1}{z_2},z_3,z_4,-\t_3,-\t_4,-\alpha_{i1},\frac{\alpha_{i2}}{z_2^2},-\alpha_{i3},-\alpha_{i4},-X_i\right),\\ \Psi_3( &x_2,x_3,x_4,y_2,y_3,y_4,z_2,z_3,z_4,\t_3,\t_4,\alpha_{i1}, \alpha_{i2}, \alpha_{i3}, \alpha_{i4}, X_i) \\
&= \left(x_2,\frac{y_3}{z_3},x_4,y_2,-\frac{x_3}{z_3},y_4,z_2,-\frac{1}{z_3},z_4,-\t_3,\t_4,\alpha_{i1},\alpha_{i2},-\frac{\alpha_{i3}}{z_3^2},\alpha_{i4},X_i\right), \\
\Psi_4(&x_2,x_3,x_4,y_2,y_3,y_4,z_2,z_3,z_4,\t_3,\t_4,\alpha_{i1}, \alpha_{i2}, \alpha_{i3}, \alpha_{i4}, X_i)\\
 &= \left(x_2,x_3,\frac{y_4}{z_4},y_2,y_3,-\frac{x_4}{z_4},z_2,z_3,-\frac{1}{z_4},\t_3,-\t_4,\alpha_{i1},\alpha_{i2},\alpha_{i3},-\frac{\alpha_{i4}}{z_4^2},X_i\right)
\end{align*}

In~\cite{Hruby_2023_CVPR}, these formulas for $\Psi_i$ had to be worked out carefully by hand.
\Cref{alg:graded-deck-interpolation} furnishes an automatic derivation.

\subsection{Nine-point Four-bar path 
generation}\label{subsec:alts-problem}

We now turn our attention to \emph{Alt's problem}.
This is a classic problem of kinematic synthesis which was first solved using homotopy continuation in work of Morgan, Sommese, and Wampler~\cite{alt-solution}.
Several more recent works have used monodromy to verify their result, eg.~\cite{DBLP:conf/imamr/HauensteinS20,plecnik,julia-example}.
Here we explain how this problem can be modeled using a branched cover, and show how its well-known symmetry group can be recovered in our approach.

The formulation we use follows~\cite{alt-solution}, employing the standard convention of \emph{isotropic coordinates}.
A vector in the plane is represented by two variables $x , \overline{x} \in \CC .$
For the purpose of solving polynomial systems, $x$ and $\overline{x}$ are treated as \emph{independent} complex variables; for any physically meaningful solutions, these coordinates will be related by complex conjugation.
With this convention, angles $T = e^{i \theta }$  are modeled by points on the hyperbola $T \overline{T} =1.$

In~\Cref{fig:alt}, the vectors $x$ and $y$ point from the coupler point $p_0$ to the upper joints of the four bar, and vectors $a$ and $b$ point from $P_0$ towards the ground pivots.
The four-bar mechanism has four revolute joints: two connecting the left ``crank'' and right ``rocker'' bars to the ground pivots, and another two connecting these bars to the base of the coupler triangle.
The motion of the mechanism is induced by rotating the crank bar about its ground pivot.
Atop the coupler triangle sits the coupler point ($\star$), which traces out a curve as the mechanism moves.
Without loss of generality, we may assume $(0,0)$ is a point on this curve.

\begin{figure}
\centering
\includegraphics[scale=0.75]{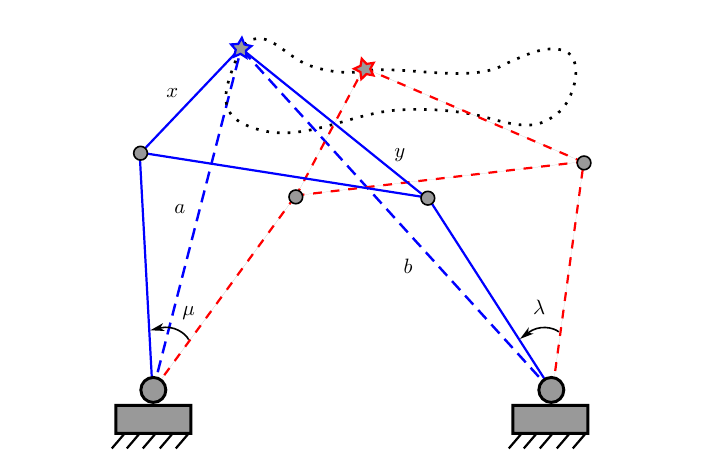}
\caption{Nine-point four-bar mechanism synthesis}
\label{fig:alt}
\end{figure}

Alt's problem can be stated as follows: given nine task positions $p_0 = 0, p_1, \ldots , p_8 \in \CC $, determine the mechanism parameters $x,y,a,b$ and angles $Q_j, T_j, S_j$ such that the coupler point moves from $p_0$ to $p_i$ for $i=1, \ldots , 8.$
Here $T_j= e^{i \lambda_j}, S_j = e^{i \mu_j }$ as in~\Cref{fig:alt}, and $Q_j = e^{i \theta_j}$ gives the rotation of the coupler point ($\star$) as it moves from $p_0$ to $p_j.$

Referring to~\Cref{fig:alt}, we may write down for each $j=1,\ldots , 8$ four loop-closure equations,
\begin{equation}\label{eq:loop-closure}
\begin{split}
Q_j (x-a) = T_j x + p_j - a, \\
S_j (y-b) = T_j y + p_j - b,
\end{split}
\end{equation}
and their conjugates.
Consequently, the orientation of the coupler point may be written as a rational function in the mechanism parameters and the other angles,
\begin{equation}\label{eq:Tj-rational}
T_j (a,b,x,y,Q_j, S_j) = (y-x)^{-1} (S_j (y-b) + Q_j (a-x) + b -a).
\end{equation}
The rocker angle $S_j$ is an algebraic function of degree $2$ in the quantities $\xx = (x,\bar{x},y,\bar{y},a,\bar{a},b,\bar{b})$ and the crank angle $Q_j$. 
That is,
\begin{equation}\label{eq:Sj-quadratic}
A(\xx, Q_j) \, S_j^2 + B(\xx, Q_j) S_j + C(\xx ,Q_j)  = 0
\end{equation}
for some $A,B,C\in \QQ [\xx , Q_j].$
We note that for generic, fixed values of mechanism parameters $\xx $, this equation defines an elliptic curve in the affine plane of $(Q_j, S_j) \in \CC^2.$
Since the discriminant of the quadratic~\eqref{eq:Sj-quadratic} is square-free, we may define an irreducible variety 
\begin{align*}
X' = \{ 
(\xx , & \, Q_1, \ldots , S_8) \in \CC^{24}
\mid \eqref{eq:Sj-quadratic},  \, \,  A(\xx , Q_j) \ne 0  \, \text{ hold,} \, \, \, j=1,\ldots , 8 \} \nonumber .
\end{align*}
Using~\eqref{eq:loop-closure}, each coupler point can now be expressed in terms of rational functions on $X',$ say $p_j (\xx , Q_j, S_j),$ and $\bar{p}_j (\xx ,  \bar{Q}_j, \bar{S}_j)$ for the conjugate.
We may then take as an irreducible variety of problem-solution pairs $X \subset \CC^8 \times \CC^{16}$ be the closed image of $X'$ under the map $(\xx , \mathbf{D}, \mathbf{Q}) \mapsto (\xx , \pp (\xx , \mathbf{Q}, \mathbf{S}), \bar{\pp} ( \xx , \mathbf{Q}, \mathbf{S} ))$.
This gives a branched cover $f:X \to \CC^{16}$. 
Although not yet formally proved, there is strong evidence that $\deg (f) = 8652.$
Following the elimination strategy described in~\cite{alt-solution}, we obtain a system of $8$ equations 
\begin{equation}\label{eq:alt-equations}
f_1 ( \xx ; \pp , \bar{\pp} ) = \ldots = f_8 ( \xx ; \pp , \bar{\pp} ) = 0
\end{equation}
that vanishes on $X$ and satisfies~\Cref{assumption:sampling-and-equations}.
With this formulation, we have two parameter-independent deck transformations: a \emph{label-swapping} that exchanges the crank and rocker bars 
\begin{equation}\label{eq:Psi-swap}
\Psi_{\text{swap}} (\xx ) = (y,\bar{y}, x , \bar{x}, b, \bar{b}, a, \bar{a}) ,
\end{equation}
(we omit the dependence of $\Psi$ on $\pp, \bar{\pp}$), 
and the \emph{Roberts cognate} map
\begin{equation}\label{eq:Psi-rob}
\Psi_{\text{Rob}} (\xx ) = \left(\displaystyle\frac{(x-a)y}{x-y}, \displaystyle\frac{(\bar{x}-\bar{a})\bar{y}}{\bar{x}-\bar{y}},  
\displaystyle\frac{bx-ay}{x-y}, \displaystyle\frac{\bar{b}\bar{x}-\bar{a}\bar{y}}{\bar{x}-\bar{y}},
a-x, \bar{a} - \bar{x},  
a, \bar{a}\right).
\end{equation}
We note that extending $\Psi_{\text{Rob}}$ to the eliminated variables $\{ Q_j, S_j, T_j \}$ yields parameter-dependent coordinate functions.


The first numerical evidence that $\deg (f) = 8652$ was given in~\cite{alt-solution}.
Later on, the lower bound $\deg (f) \ge 8652$ was certified by Hauenstein and Sottile using Smale's $\alpha$-theory~\cite{alpha-certified}.
A rigorous proof that this bound is tight remains an open problem.
More recently, Sottile and Yahl have posed the problem of determining the Galois/monodromy group of the branched cover $f$~\cite[\S 7.3]{galois-survey}.
From equations~\eqref{eq:alt-equations}, we heuristically computed permutations in $\mon (f)$ using monodromy loops. 
This produced 4 permutations using default settings.
Using~\Cref{prop:centralizer}, we determine that the deck transformation group is isomorphic to $S_3,$ generated by a transposition and 3-cycle corresponding to~\eqref{eq:Psi-swap} and~\eqref{eq:Psi-rob}, respectively.

We remark that, in this formulation of Alt's problem, the methods of~\Cref{sec:symmetries} detect that the group of scaling symmetries $\Gcon \times \Gdis $ is trivial.
Thus,~\Cref{alg:graded-deck-interpolation} yields no improvement over~\Cref{alg:deck-interpolation} in this case.
It would, however, be interesting to investigate these symmetries for other formulations of Alt's problem. 

\begin{table}[H]
    \footnotesize
    \centering
    \begin{tabular}{l c c} \toprule
         & Inhom & Inhom \\
         & $\&$ & $\&$ \\
         & Param-indep & Param-dep \\
        \midrule
        Number of solutions & 8652 & 8652 \\
        Degree bound $D^*$ & 2 & 2 \\
        Monodromy time & 15min & 15min  \\
        Tracking $\&$ Interpolation time &  $<1$sec & 10sec \\
        Largest Vandermonde matrix & $90 \times 90$ & $650 \times 650$ \\
        Number of tracked paths & $6 \times 90$ & $6 \times 650$ \\
        $\Psi_{\text{swap}}, \Psi_{\text{Rob}}$ fully recovered? & yes & yes \\
        \bottomrule
    \end{tabular}
    \caption{Timings/details for recovering deck transformations of the Alt's problem.}
    \label{tab:alt}
\end{table}

As we can see from~\Cref{tab:alt}, running both parameter-independent and parameter-dependent formulations using~\Cref{alg:deck-interpolation} yield similar results, where both successfully interpolate the deck transformations.  For this example, ~\Cref{alg:graded-deck-interpolation} produces the same results as~\Cref{alg:deck-interpolation} and thus is omitted from~\Cref{tab:alt}.  



\section{Conclusion}\label{sec:conclusion}

In summary, we have proposed a novel method for recovering hidden symmetries of commonly-occuring parametric polynomial systems.
Despite its heuristic nature, our experiments demonstrate that the method is capable of delivering results, even on examples with a relatively larger number of solutions like~\Cref{subsec:radial,subsec:alts-problem}.
Additionally, we have introduced a novel quasi-homgeneous interpolation framework in~\Cref{sec:symmetries}, delivering much-improved results compared to the baselines in our conference paper~\cite{us_ISSAC23}. 

One obvious avenue for further research is to test more examples and develop better heuristics.
It would be highly of interest to develop more robust numerical interpolation methods, eg.~by exploiting bases other than the standard monomials, and test them on examples similar to those in~\Cref{sec:ex}.
There is also potential for fruitful contact with more traditional methods of symbolic computation.
In addition to our comments in~\Cref{sec:previous-works}, we point out that some hybrid symbolic-numerical methods may be useful in practice.
For instance, it seems plausible that one could (1) run~\Cref{alg:deck-interpolation} or~\ref{alg:graded-deck-interpolation} until recovering coordinate functions for the deck transformations on some subset of variables $\yy \subset \xx $, then (2) eliminate the remaining variables $\xx \setminus \yy $ and use parametric Gr\"{o}bner bases to solve for their coordinate functions using the interpolated expressions from step (1).
Such hybrid methods might also be useful for recovering deck transformations when a decomposition as in~\Cref{def:decomposable-branched-cover} is already known, or vice-versa.

As noted in~\Cref{sec:background}, decomposability and the existence deck transformations are closely related, but not equivalent.
Indeed, the radial camera relative pose problem of~\Cref{subsec:radial} has a deck transformation group of order $16,$ which implies the degree $3584$ branched cover associated to this problem decomposes as a composition of covers of degree $16$ and $224.$
The latter cover, despite not having any deck transformations, does in fact decompose further into covers of degree $4,2$ and $28.$
Developing numerical methods for determining the maps and intermediate varieties appearing in such a decomposition is a highly appealing next step.

\section*{Acknowledgements}

\noindent
T.~Duff acknowledges support from NSF DMS-2103310. V. Korotynskiy and T. Pajdla acknowledge support from EU H2020 No. 871245 SPRING project. V. Korotynskiy was partially supported by the Grant Agency of CTU in Prague project SGS23/056/OHK3/1T/13.
We thank Taylor Brysiewicz for helpful conversations that got us up to speed on Julia package development.



\bibliographystyle{elsarticle-num-names}\biboptions{sort} 
\bibliography{main}

\begin{thebibliography}{46}
\expandafter\ifx\csname natexlab\endcsname\relax\def\natexlab#1{#1}\fi
\providecommand{\url}[1]{\texttt{#1}}
\providecommand{\href}[2]{#2}
\providecommand{\path}[1]{#1}
\providecommand{\DOIprefix}{doi:}
\providecommand{\ArXivprefix}{arXiv:}
\providecommand{\URLprefix}{URL: }
\providecommand{\Pubmedprefix}{pmid:}
\providecommand{\doi}[1]{\href{http://dx.doi.org/#1}{\path{#1}}}
\providecommand{\Pubmed}[1]{\href{pmid:#1}{\path{#1}}}
\providecommand{\bibinfo}[2]{#2}
\ifx\xfnm\relax \def\xfnm[#1]{\unskip,\space#1}\fi
\bibitem[{Sottile and Yahl(2021)}]{galois-survey}
\bibinfo{author}{F.~Sottile}, \bibinfo{author}{T.~Yahl},
\newblock \bibinfo{title}{Galois groups in enumerative geometry and
  applications}  (\bibinfo{year}{2021}). \URLprefix
  \url{https://arxiv.org/abs/2108.07905}.
  \DOIprefix\doi{10.48550/ARXIV.2108.07905}.
\bibitem[{Hauenstein et~al.(2018)Hauenstein, Rodriguez, and
  Sottile}]{NumGalois}
\bibinfo{author}{J.~D. Hauenstein}, \bibinfo{author}{J.~I. Rodriguez},
  \bibinfo{author}{F.~Sottile},
\newblock \bibinfo{title}{Numerical computation of {G}alois groups},
\newblock \bibinfo{journal}{Found. Comput. Math.} \bibinfo{volume}{18}
  (\bibinfo{year}{2018}) \bibinfo{pages}{867--890}. \URLprefix
  \url{https://doi.org/10.1007/s10208-017-9356-x}.
  \DOIprefix\doi{10.1007/s10208-017-9356-x}.
\bibitem[{Duff et~al.(2019)Duff, Hill, Jensen, Lee, Leykin, and
  Sommars}]{duff-monodromy}
\bibinfo{author}{T.~Duff}, \bibinfo{author}{C.~Hill},
  \bibinfo{author}{A.~Jensen}, \bibinfo{author}{K.~Lee},
  \bibinfo{author}{A.~Leykin}, \bibinfo{author}{J.~Sommars},
\newblock \bibinfo{title}{Solving polynomial systems via homotopy continuation
  and monodromy},
\newblock \bibinfo{journal}{IMA Journal of Numerical Analysis}
  \bibinfo{volume}{39} (\bibinfo{year}{2019}) \bibinfo{pages}{1421--1446}.
\bibitem[{Duff et~al.(2023)Duff, Korotynskiy, Pajdla, and Regan}]{us_ISSAC23}
\bibinfo{author}{T.~Duff}, \bibinfo{author}{V.~Korotynskiy},
  \bibinfo{author}{T.~Pajdla}, \bibinfo{author}{M.~H. Regan},
\newblock \bibinfo{title}{Using monodromy to recover symmetries of polynomial
  systems},
\newblock in: \bibinfo{booktitle}{I{SSAC}'23---{P}roceedings of the 2023 {ACM}
  {I}nternational {S}ymposium on {S}ymbolic and {A}lgebraic {C}omputation},
  \bibinfo{publisher}{ACM, New York}, \bibinfo{year}{2023}, pp.
  \bibinfo{pages}{251--259}.
\bibitem[{Bezanson et~al.(2017)Bezanson, Edelman, Karpinski, and Shah}]{julia}
\bibinfo{author}{J.~Bezanson}, \bibinfo{author}{A.~Edelman},
  \bibinfo{author}{S.~Karpinski}, \bibinfo{author}{V.~B. Shah},
\newblock \bibinfo{title}{Julia: A fresh approach to numerical computing},
\newblock \bibinfo{journal}{SIAM {R}eview} \bibinfo{volume}{59}
  (\bibinfo{year}{2017}) \bibinfo{pages}{65--98}. \URLprefix
  \url{https://epubs.siam.org/doi/10.1137/141000671}.
  \DOIprefix\doi{10.1137/141000671}.
\bibitem[{Galligo and Poteaux(2009)}]{GalligoPoteaux}
\bibinfo{author}{A.~Galligo}, \bibinfo{author}{A.~Poteaux},
\newblock \bibinfo{title}{Continuations and monodromy on random riemann
  surfaces},
\newblock in: \bibinfo{booktitle}{Proceedings of the 2009 Conference on
  Symbolic Numeric Computation}, SNC '09, \bibinfo{publisher}{Association for
  Computing Machinery}, \bibinfo{address}{New York, NY, USA},
  \bibinfo{year}{2009}, p. \bibinfo{pages}{115–124}. \URLprefix
  \url{https://doi.org/10.1145/1577190.1577210}.
  \DOIprefix\doi{10.1145/1577190.1577210}.
\bibitem[{Deconinck and van Hoeij(2001)}]{VanHoeij}
\bibinfo{author}{B.~Deconinck}, \bibinfo{author}{M.~van Hoeij},
\newblock \bibinfo{title}{Computing {R}iemann matrices of algebraic curves},
\newblock volume \bibinfo{volume}{152/153}, \bibinfo{year}{2001}, pp.
  \bibinfo{pages}{28--46}. \URLprefix
  \url{https://doi.org/10.1016/S0167-2789(01)00156-7}.
  \DOIprefix\doi{10.1016/S0167-2789(01)00156-7}, \bibinfo{note}{advances in
  nonlinear mathematics and science}.
\bibitem[{Sommese et~al.(2001)Sommese, Verschelde, and Wampler}]{NID}
\bibinfo{author}{A.~J. Sommese}, \bibinfo{author}{J.~Verschelde},
  \bibinfo{author}{C.~W. Wampler},
\newblock \bibinfo{title}{Numerical decomposition of the solution sets of
  polynomial systems into irreducible components},
\newblock \bibinfo{journal}{SIAM J. Numer. Anal.} \bibinfo{volume}{38}
  (\bibinfo{year}{2001}) \bibinfo{pages}{2022--2046}. \URLprefix
  \url{https://doi.org/10.1137/S0036142900372549}.
  \DOIprefix\doi{10.1137/S0036142900372549}.
\bibitem[{Mart\'{\i}n~del Campo and Rodriguez(2017)}]{Critical}
\bibinfo{author}{A.~Mart\'{\i}n~del Campo}, \bibinfo{author}{J.~I. Rodriguez},
\newblock \bibinfo{title}{Critical points via monodromy and local methods},
\newblock \bibinfo{journal}{J. Symbolic Comput.} \bibinfo{volume}{79}
  (\bibinfo{year}{2017}) \bibinfo{pages}{559--574}. \URLprefix
  \url{https://doi.org/10.1016/j.jsc.2016.07.019}.
  \DOIprefix\doi{10.1016/j.jsc.2016.07.019}.
\bibitem[{Am\'{e}ndola et~al.(2016)Am\'{e}ndola, Lindberg, and
  Rodriguez}]{Amendola}
\bibinfo{author}{C.~Am\'{e}ndola}, \bibinfo{author}{J.~Lindberg},
  \bibinfo{author}{J.~I. Rodriguez}, \bibinfo{title}{Solving parameterized
  polynomial systems with decomposable projections}, \bibinfo{year}{2016}.
  \URLprefix \url{https://arxiv.org/abs/1612.08807}.
\bibitem[{Brysiewicz et~al.(2021)Brysiewicz, Rodriguez, Sottile, and
  Yahl}]{yahl}
\bibinfo{author}{T.~Brysiewicz}, \bibinfo{author}{J.~I. Rodriguez},
  \bibinfo{author}{F.~Sottile}, \bibinfo{author}{T.~Yahl},
\newblock \bibinfo{title}{Solving decomposable sparse systems},
\newblock \bibinfo{journal}{Numer. Algorithms} \bibinfo{volume}{88}
  (\bibinfo{year}{2021}) \bibinfo{pages}{453--474}. \URLprefix
  \url{https://doi.org/10.1007/s11075-020-01045-x}.
  \DOIprefix\doi{10.1007/s11075-020-01045-x}.
\bibitem[{Duff et~al.(2022)Duff, Korotynskiy, Pajdla, and
  Regan}]{GaloisComputerVision}
\bibinfo{author}{T.~Duff}, \bibinfo{author}{V.~Korotynskiy},
  \bibinfo{author}{T.~Pajdla}, \bibinfo{author}{M.~H. Regan},
\newblock \bibinfo{title}{Galois/monodromy groups for decomposing minimal
  problems in 3d reconstruction},
\newblock \bibinfo{journal}{SIAM Journal on Applied Algebra and Geometry}
  \bibinfo{volume}{6} (\bibinfo{year}{2022}) \bibinfo{pages}{740--772}.
  \URLprefix \url{https://doi.org/10.1137/21M1422872}.
  \DOIprefix\doi{10.1137/21M1422872}.
\bibitem[{Kaltofen and Yang(2007)}]{DBLP:conf/issac/KaltofenY07}
\bibinfo{author}{E.~Kaltofen}, \bibinfo{author}{Z.~Yang},
\newblock \bibinfo{title}{On exact and approximate interpolation of sparse
  rational functions},
\newblock in: \bibinfo{editor}{D.~Wang} (Ed.), \bibinfo{booktitle}{Symbolic and
  Algebraic Computation, International Symposium, {ISSAC} 2007, Waterloo,
  Ontario, Canada, July 28 - August 1, 2007, Proceedings},
  \bibinfo{publisher}{{ACM}}, \bibinfo{year}{2007}, pp.
  \bibinfo{pages}{203--210}. \URLprefix
  \url{https://doi.org/10.1145/1277548.1277577}.
  \DOIprefix\doi{10.1145/1277548.1277577}.
\bibitem[{Cuyt and Lee(2011)}]{DBLP:journals/tcs/CuytL11}
\bibinfo{author}{A.~A.~M. Cuyt}, \bibinfo{author}{W.~Lee},
\newblock \bibinfo{title}{Sparse interpolation of multivariate rational
  functions},
\newblock \bibinfo{journal}{Theor. Comput. Sci.} \bibinfo{volume}{412}
  (\bibinfo{year}{2011}) \bibinfo{pages}{1445--1456}. \URLprefix
  \url{https://doi.org/10.1016/j.tcs.2010.11.050}.
  \DOIprefix\doi{10.1016/j.tcs.2010.11.050}.
\bibitem[{van~der Hoeven and Lecerf(2021)}]{DBLP:journals/cca/HoevenL21}
\bibinfo{author}{J.~van~der Hoeven}, \bibinfo{author}{G.~Lecerf},
\newblock \bibinfo{title}{On sparse interpolation of rational functions and
  gcds},
\newblock \bibinfo{journal}{{ACM} Commun. Comput. Algebra} \bibinfo{volume}{55}
  (\bibinfo{year}{2021}) \bibinfo{pages}{1--12}. \URLprefix
  \url{https://doi.org/10.1145/3466895.3466896}.
  \DOIprefix\doi{10.1145/3466895.3466896}.
\bibitem[{Xu et~al.(2018)Xu, Burr, and Yap}]{xu-burr-yap}
\bibinfo{author}{J.~Xu}, \bibinfo{author}{M.~Burr}, \bibinfo{author}{C.~Yap},
\newblock \bibinfo{title}{An approach for certifying homotopy continuation
  paths: univariate case},
\newblock in: \bibinfo{booktitle}{I{SSAC}'18---{P}roceedings of the 2018 {ACM}
  {I}nternational {S}ymposium on {S}ymbolic and {A}lgebraic {C}omputation},
  \bibinfo{publisher}{ACM, New York}, \bibinfo{year}{2018}, pp.
  \bibinfo{pages}{399--406}. \URLprefix
  \url{https://doi.org/10.1145/3208976.3209010}.
  \DOIprefix\doi{10.1145/3208976.3209010}.
\bibitem[{Beltr\'{a}n and Leykin(2013)}]{berltran-leykin}
\bibinfo{author}{C.~Beltr\'{a}n}, \bibinfo{author}{A.~Leykin},
\newblock \bibinfo{title}{Robust certified numerical homotopy tracking},
\newblock \bibinfo{journal}{Found. Comput. Math.} \bibinfo{volume}{13}
  (\bibinfo{year}{2013}) \bibinfo{pages}{253--295}. \URLprefix
  \url{https://doi.org/10.1007/s10208-013-9143-2}.
  \DOIprefix\doi{10.1007/s10208-013-9143-2}.
\bibitem[{Hauenstein et~al.(2014)Hauenstein, Haywood, and
  Liddell}]{hauenstein-liddell-haywood}
\bibinfo{author}{J.~D. Hauenstein}, \bibinfo{author}{I.~Haywood},
  \bibinfo{author}{A.~C. Liddell, Jr.},
\newblock \bibinfo{title}{An {\it a posteriori} certification algorithm for
  {N}ewton homotopies}  (\bibinfo{year}{2014}) \bibinfo{pages}{248--255}.
  \URLprefix \url{https://doi.org/10.1145/2608628.2608651}.
  \DOIprefix\doi{10.1145/2608628.2608651}.
\bibitem[{van~der Hoeven(2011)}]{vdH:homotopy}
\bibinfo{author}{J.~van~der Hoeven}, \bibinfo{title}{Reliable homotopy
  continuation}, \bibinfo{type}{Technical Report}, HAL, \bibinfo{year}{2011}.
  \bibinfo{note}{Http://hal.archives-ouvertes.fr/hal-00589948/fr/}.
\bibitem[{Brysiewicz et~al.(2021)Brysiewicz, Rodriguez, Sottile, and
  Yahl}]{yahl2}
\bibinfo{author}{T.~Brysiewicz}, \bibinfo{author}{J.~I. Rodriguez},
  \bibinfo{author}{F.~Sottile}, \bibinfo{author}{T.~Yahl},
\newblock \bibinfo{title}{Decomposable sparse polynomial systems},
\newblock \bibinfo{journal}{Journal of Software for Algebra and Geometry}
  \bibinfo{volume}{11} (\bibinfo{year}{2021}) \bibinfo{pages}{53--59}.
\bibitem[{Larsson and Åström(2016)}]{Larssonsymm}
\bibinfo{author}{V.~Larsson}, \bibinfo{author}{K.~Åström},
\newblock \bibinfo{title}{Uncovering symmetries in polynomial systems},
\newblock in: \bibinfo{editor}{B.~Leibe}, \bibinfo{editor}{J.~Matas},
  \bibinfo{editor}{N.~Sebe}, \bibinfo{editor}{M.~Welling} (Eds.),
  \bibinfo{booktitle}{Computer Vision -- ECCV 2016},
  \bibinfo{publisher}{Springer International Publishing}, \bibinfo{year}{2016},
  pp. \bibinfo{pages}{252--267}.
\bibitem[{Corless et~al.(2009)Corless, Gatermann, and Kotsireas}]{Corless}
\bibinfo{author}{R.~M. Corless}, \bibinfo{author}{K.~Gatermann},
  \bibinfo{author}{I.~S. Kotsireas},
\newblock \bibinfo{title}{Using symmetries in the eigenvalue method for
  polynomial systems},
\newblock \bibinfo{journal}{J. Symbolic Comput.} \bibinfo{volume}{44}
  (\bibinfo{year}{2009}) \bibinfo{pages}{1536--1550}. \URLprefix
  \url{https://doi.org/10.1016/j.jsc.2008.11.009}.
  \DOIprefix\doi{10.1016/j.jsc.2008.11.009}.
\bibitem[{Hubert and Labahn(2013)}]{Hubert1}
\bibinfo{author}{E.~Hubert}, \bibinfo{author}{G.~Labahn},
\newblock \bibinfo{title}{Scaling invariants and symmetry reduction of
  dynamical systems},
\newblock \bibinfo{journal}{Foundations of Computational Mathematics}
  \bibinfo{volume}{13} (\bibinfo{year}{2013}) \bibinfo{pages}{479--516}.
\bibitem[{Hubert and Labahn(2016)}]{Hubert2}
\bibinfo{author}{E.~Hubert}, \bibinfo{author}{G.~Labahn},
\newblock \bibinfo{title}{Computation of invariants of finite abelian groups},
\newblock \bibinfo{journal}{Math. Comp.} \bibinfo{volume}{85}
  (\bibinfo{year}{2016}) \bibinfo{pages}{3029--3050}.
\bibitem[{Hatcher(2002)}]{hatcher}
\bibinfo{author}{A.~Hatcher}, \bibinfo{title}{Algebraic topology},
  \bibinfo{publisher}{Cambridge University Press},
  \bibinfo{address}{Cambridge}, \bibinfo{year}{2002}.
\bibitem[{Miranda(1995)}]{Miranda}
\bibinfo{author}{R.~Miranda}, \bibinfo{title}{Algebraic curves and {R}iemann
  surfaces}, volume~\bibinfo{volume}{5} of \textit{\bibinfo{series}{Graduate
  Studies in Mathematics}}, \bibinfo{publisher}{American Mathematical Society,
  Providence, RI}, \bibinfo{year}{1995}. \URLprefix
  \url{https://doi.org/10.1090/gsm/005}. \DOIprefix\doi{10.1090/gsm/005}.
\bibitem[{Ritt(1922)}]{MR1501205}
\bibinfo{author}{J.~F. Ritt},
\newblock \bibinfo{title}{Errata: ``{P}rime and composite polynomials''
  [{T}rans. {A}mer. {M}ath. {S}oc. {\bf 23} (1922), no. 1, 51--66; 1501189]},
\newblock \bibinfo{journal}{Trans. Amer. Math. Soc.} \bibinfo{volume}{23}
  (\bibinfo{year}{1922}) \bibinfo{pages}{431}. \URLprefix
  \url{https://doi.org/10.2307/1988887}. \DOIprefix\doi{10.2307/1988887}.
\bibitem[{Faug{\`{e}}re et~al.(2010)Faug{\`{e}}re, von~zur Gathen, and
  Perret}]{DBLP:conf/issac/FaugereGP10}
\bibinfo{author}{J.~Faug{\`{e}}re}, \bibinfo{author}{J.~von~zur Gathen},
  \bibinfo{author}{L.~Perret},
\newblock \bibinfo{title}{Decomposition of generic multivariate polynomials},
\newblock in: \bibinfo{editor}{W.~Koepf} (Ed.), \bibinfo{booktitle}{Symbolic
  and Algebraic Computation, International Symposium, {ISSAC} 2010, Munich,
  Germany, July 25-28, 2010, Proceedings}, \bibinfo{publisher}{{ACM}},
  \bibinfo{year}{2010}, pp. \bibinfo{pages}{131--137}. \URLprefix
  \url{https://doi.org/10.1145/1837934.1837963}.
  \DOIprefix\doi{10.1145/1837934.1837963}.
\bibitem[{von~zur Gathen et~al.(1999)von~zur Gathen, Gutierrez, and
  Rubio}]{DBLP:conf/casc/Gathen0R99}
\bibinfo{author}{J.~von~zur Gathen}, \bibinfo{author}{J.~Gutierrez},
  \bibinfo{author}{R.~Rubio},
\newblock \bibinfo{title}{On multivariate polynomial decomposition},
\newblock in: \bibinfo{editor}{V.~G. Ganzha}, \bibinfo{editor}{E.~W. Mayr},
  \bibinfo{editor}{E.~V. Vorozhtsov} (Eds.), \bibinfo{booktitle}{Proceedings of
  the Second Workshop on Computer Algebra in Scientific Computing, {CASC} 1999,
  Munich, Germany, May 31 - June 4, 1999}, \bibinfo{publisher}{Springer},
  \bibinfo{year}{1999}, pp. \bibinfo{pages}{463--478}. \URLprefix
  \url{https://doi.org/10.1007/978-3-642-60218-4\_35}.
  \DOIprefix\doi{10.1007/978-3-642-60218-4\_35}.
\bibitem[{Coleman and Pothen(1986)}]{nullspace-hard}
\bibinfo{author}{T.~F. Coleman}, \bibinfo{author}{A.~Pothen},
\newblock \bibinfo{title}{The null space problem. {I}. {C}omplexity},
\newblock \bibinfo{journal}{SIAM J. Algebraic Discrete Methods}
  \bibinfo{volume}{7} (\bibinfo{year}{1986}) \bibinfo{pages}{527--537}.
  \URLprefix \url{https://doi.org/10.1137/0607059}.
  \DOIprefix\doi{10.1137/0607059}.
\bibitem[{Vinberg and Popov(1989)}]{vinberg-popov}
\bibinfo{author}{E.~B. Vinberg}, \bibinfo{author}{V.~L. Popov},
\newblock \bibinfo{title}{Invariant theory},
\newblock in: \bibinfo{booktitle}{Algebraic geometry, 4 ({R}ussian)}, Itogi
  Nauki i Tekhniki, \bibinfo{publisher}{Akad. Nauk SSSR, Vsesoyuz. Inst.
  Nauchn. i Tekhn. Inform., Moscow}, \bibinfo{year}{1989}, pp.
  \bibinfo{pages}{137--314, 315}.
\bibitem[{Grayson and Stillman(????)}]{M2}
\bibinfo{author}{D.~R. Grayson}, \bibinfo{author}{M.~E. Stillman},
  \bibinfo{title}{Macaulay2, a software system for research in algebraic
  geometry}, \bibinfo{howpublished}{Available at
  \url{http://www2.macaulay2.com}}, ????
\bibitem[{Morgan and Sommese(1989)}]{coeffParam}
\bibinfo{author}{A.~P. Morgan}, \bibinfo{author}{A.~J. Sommese},
\newblock \bibinfo{title}{Coefficient-parameter polynomial continuation},
\newblock \bibinfo{journal}{Appl. Math. Comput.} \bibinfo{volume}{29}
  (\bibinfo{year}{1989}) \bibinfo{pages}{123--160}. \URLprefix
  \url{https://doi.org/10.1016/0096-3003(89)90099-4}.
  \DOIprefix\doi{10.1016/0096-3003(89)90099-4}.
\bibitem[{Breiding and Timme(2018)}]{HCJL}
\bibinfo{author}{P.~Breiding}, \bibinfo{author}{S.~Timme},
\newblock \bibinfo{title}{Homotopy{C}ontinuation.jl: {A} {P}ackage for
  {H}omotopy {C}ontinuation in {J}ulia},
\newblock in: \bibinfo{booktitle}{Mathematical Software -- ICMS 2018},
  \bibinfo{publisher}{Springer International Publishing},
  \bibinfo{address}{Cham}, \bibinfo{year}{2018}, pp. \bibinfo{pages}{458--465}.
\bibitem[{Fieker et~al.(2017)Fieker, Hart, Hofmann, and Johansson}]{nemo}
\bibinfo{author}{C.~Fieker}, \bibinfo{author}{W.~Hart},
  \bibinfo{author}{T.~Hofmann}, \bibinfo{author}{F.~Johansson},
\newblock \bibinfo{title}{Nemo/hecke: computer algebra and number theory
  packages for the julia programming language},
\newblock in: \bibinfo{booktitle}{Proceedings of the 2017 acm on international
  symposium on symbolic and algebraic computation}, \bibinfo{year}{2017}, pp.
  \bibinfo{pages}{157--164}.
\bibitem[{OSCAR(2023)}]{OSCAR}
OSCAR, \bibinfo{title}{Oscar -- open source computer algebra research system,
  version 0.11.3-dev}, \bibinfo{year}{2023}. \URLprefix
  \url{https://oscar.computeralgebra.de}.
\bibitem[{Nist{\'{e}}r(2004)}]{nister}
\bibinfo{author}{D.~Nist{\'{e}}r},
\newblock \bibinfo{title}{An efficient solution to the five-point relative pose
  problem},
\newblock \bibinfo{journal}{{IEEE} Trans. Pattern Anal. Mach. Intell.}
  \bibinfo{volume}{26} (\bibinfo{year}{2004}) \bibinfo{pages}{756--777}.
  \URLprefix \url{https://doi.org/10.1109/TPAMI.2004.17}.
  \DOIprefix\doi{10.1109/TPAMI.2004.17}.
\bibitem[{Sturm(2011)}]{sturm2011historical}
\bibinfo{author}{P.~Sturm},
\newblock \bibinfo{title}{A historical survey of geometric computer vision},
\newblock in: \bibinfo{booktitle}{International Conference on Computer Analysis
  of Images and Patterns}, \bibinfo{organization}{Springer},
  \bibinfo{year}{2011}, pp. \bibinfo{pages}{1--8}.
\bibitem[{Pan(2016)}]{Pan_Vandermonde}
\bibinfo{author}{V.~Y. Pan},
\newblock \bibinfo{title}{How bad are vandermonde matrices?},
\newblock \bibinfo{journal}{SIAM Journal on Matrix Analysis and Applications}
  \bibinfo{volume}{37} (\bibinfo{year}{2016}) \bibinfo{pages}{676--694}.
  \DOIprefix\doi{10.1137/15M1030170}.
\bibitem[{Higham(1987)}]{Higham}
\bibinfo{author}{N.~Higham},
\newblock \bibinfo{title}{Error analysis of the bj{\"o}rck-pereyra algorithms
  for solving vandermonde systems},
\newblock \bibinfo{journal}{NUMERISCHE MATHEMATIK} \bibinfo{volume}{50}
  (\bibinfo{year}{1987}) \bibinfo{pages}{613--632}.
  \DOIprefix\doi{10.1007/BF01408579}.
\bibitem[{Hruby et~al.(2023)Hruby, Korotynskiy, Duff, Oeding, Pollefeys,
  Pajdla, and Larsson}]{Hruby_2023_CVPR}
\bibinfo{author}{P.~Hruby}, \bibinfo{author}{V.~Korotynskiy},
  \bibinfo{author}{T.~Duff}, \bibinfo{author}{L.~Oeding},
  \bibinfo{author}{M.~Pollefeys}, \bibinfo{author}{T.~Pajdla},
  \bibinfo{author}{V.~Larsson},
\newblock \bibinfo{title}{Four-view geometry with unknown radial distortion},
\newblock in: \bibinfo{booktitle}{Proceedings of the IEEE/CVF Conference on
  Computer Vision and Pattern Recognition (CVPR)}, \bibinfo{year}{2023}, pp.
  \bibinfo{pages}{8990--9000}.
\bibitem[{Wampler et~al.(1992)Wampler, Morgan, and Sommese}]{alt-solution}
\bibinfo{author}{C.~W. Wampler}, \bibinfo{author}{A.~P. Morgan},
  \bibinfo{author}{A.~J. Sommese},
\newblock \bibinfo{title}{{Complete Solution of the Nine-Point Path Synthesis
  Problem for Four-Bar Linkages}},
\newblock \bibinfo{journal}{Journal of Mechanical Design} \bibinfo{volume}{114}
  (\bibinfo{year}{1992}) \bibinfo{pages}{153--159}. \URLprefix
  \url{https://doi.org/10.1115/1.2916909}. \DOIprefix\doi{10.1115/1.2916909}.
  \href{http://arxiv.org/abs/https://asmedigitalcollection.asme.org/mechanicaldesign/article-pdf/114/1/153/5507001/153\_1.pdf}{{\tt
  arXiv:https://asmedigitalcollection.asme.org/mechanicaldesign/article-pdf/114/1/153/5507001/153\_1.pdf}}.
\bibitem[{Hauenstein and Sherman(2020)}]{DBLP:conf/imamr/HauensteinS20}
\bibinfo{author}{J.~D. Hauenstein}, \bibinfo{author}{S.~N. Sherman},
\newblock \bibinfo{title}{Using monodromy to statistically estimate the number
  of solutions},
\newblock in: \bibinfo{editor}{W.~Holderbaum}, \bibinfo{editor}{J.~M. Selig}
  (Eds.), \bibinfo{booktitle}{2nd {IMA} Conference on Mathematics of Robotics,
  Manchester, UK, 9-11 September 2020}, volume~\bibinfo{volume}{21} of
  \textit{\bibinfo{series}{Springer Proceedings in Advanced Robotics}},
  \bibinfo{publisher}{Springer}, \bibinfo{year}{2020}, pp.
  \bibinfo{pages}{37--46}. \URLprefix
  \url{https://doi.org/10.1007/978-3-030-91352-6\_4}.
  \DOIprefix\doi{10.1007/978-3-030-91352-6\_4}.
\bibitem[{Plecnik and Fearing(2017)}]{plecnik}
\bibinfo{author}{M.~M. Plecnik}, \bibinfo{author}{R.~S. Fearing},
\newblock \bibinfo{title}{Finding only finite roots to large kinematic
  synthesis systems},
\newblock \bibinfo{journal}{Journal of Mechanisms and Robotics}
  \bibinfo{volume}{9} (\bibinfo{year}{2017}) \bibinfo{pages}{021005}.
\bibitem[{Paul~Breiding and Timme(????)}]{julia-example}
\bibinfo{author}{A.~E. Paul~Breiding}, \bibinfo{author}{S.~Timme},
  \bibinfo{title}{Alt's problem}, \bibinfo{howpublished}{\url{
  https://www.JuliaHomotopyContinuation.org/examples/alts-problem/ }}, ????
  \bibinfo{note}{Accessed: June 27, 2022}.
\bibitem[{Hauenstein and Sottile(2012)}]{alpha-certified}
\bibinfo{author}{J.~D. Hauenstein}, \bibinfo{author}{F.~Sottile},
\newblock \bibinfo{title}{Algorithm 921: alpha{C}ertified: certifying solutions
  to polynomial systems},
\newblock \bibinfo{journal}{ACM Trans. Math. Software} \bibinfo{volume}{38}
  (\bibinfo{year}{2012}) \bibinfo{pages}{Art. 28, 20}. \URLprefix
  \url{https://doi.org/10.1145/2331130.2331136}.
  \DOIprefix\doi{10.1145/2331130.2331136}.

\end{thebibliography}





\end{document}